\documentclass[12pt]{article}
\usepackage{amssymb}

\usepackage{amsmath}
\usepackage{mitpress}

\newtheorem{theorem}{Theorem}

\newtheorem{condition}[theorem]{Condition}

\newtheorem{proposition}[theorem]{Proposition}
\newtheorem{remark}[theorem]{Remark}

\newenvironment{proof}[1][Proof]{\noindent\textbf{#1.} }{\ \rule{0.5em}{0.5em}}
\newdimen\dummy
\dummy=\oddsidemargin
\begin{document}

\title{Spin Needlets Spectral Estimation }
\author{Daryl Geller$^{\ast }$, Xiaohong Lan$^{\ast \ast }$ and Domenico
Marinucci$^{\ast \ast }$ \\
%EndAName
* Department of Mathematics, Stony Brook University and \\
** Department of Mathematics, University of Rome Tor Vergata}
\maketitle

\begin{abstract}
We consider the statistical analysis of random sections of a
\emph{spin fibre bundle }over the sphere. These may be thought of as
random fields that at each point $p\in \mathbb{S}^{2}$ take as a
value a curve (e.g. an ellipse) living in the tangent plane at that
point $T_{p}\mathbb{S}^{2},$ rather than a number as in ordinary
situations$.$ The analysis of such fields is strongly motivated by
applications, for instance polarization experiments in Cosmology. To
investigate such fields, spin needlets were recently introduced by
\cite{gelmar} and \cite{ghmkp}. We consider the use of spin needlets
for spin angular power spectrum estimation, in the presence of noise
and missing observations, and we provide Central Limit Theorem
results, in the high frequency sense; we discuss also tests for bias
and asymmetries with an asymptotic justification.

\begin{itemize}
\item Keywords and Phrases: Spin Random Fields, Spin Needlets, CMB
Polarization, Angular Power Spectrum Estimation, Fibre Bundles

\end{itemize}
\end{abstract}

\section{Introduction}

The analysis of (random or deterministic) functions defined on the sphere by
means of wavelets has recently been the object of a number of theoretical
and applied papers, see for instance \cite{antoine}, \cite{antodem}, \cite%
{jfaa1,jfaa3}, \cite{npw1,npw2}, \cite{gm1,gm2,gm3}, \cite{bkmpAoS,bkmpBer}, %
\cite{guilloux}. Many of these works have found their motivating rationale
in recent developments in the applied sciences, such as Medical Imaging,
Geophysics, Atmospheric Sciences, Astrophysics and Cosmology. These same
fields of applications are now prompting stochastic models which are more
sophisticated (and more intriguing) than ordinary, scalar valued random
fields. In this paper, we shall be especially concerned with astrophysical
and cosmological applications, but several similar issues can be found in
other disciplines, see for instance \cite{schwartzman} for related
mathematical models in the field of brain mapping.

Concerning astrophysics, there are now many mathematical papers which have
been motivated by the analysis of so-called Cosmic Microwave Background
radiation (CMB); the latter can be very loosely viewed as a relic
electromagnetic radiation which permeates the Universe providing a map of
its status from 13.7 billions ago, in the immediate adjacency of the Big
Bang. Almost all mathematical statistics papers in this area have been
concerned with the temperature component of CMB, which can be represented as
a standard spherical random field (see \cite{cmaas} for a review). We recall
that a scalar random field on the sphere may be thought of as a collection
of random variables $\left\{ T(p):p\in \mathbb{S}^{2}\right\} $, where $%
\mathbb{S}^{2}=\left\{ p:\left\| p\right\| ^{2}=1\right\} $ is the unit
sphere of $\mathbb{R}^{3}$ and $\left\| .\right\| $ denotes Euclidean norm. $%
T(p)$ is isotropic if its law is invariant with respect to the group of
rotations, $T(p)\overset{d}{=}T(gp)$ for all $g\in SO(3),$ where $\overset{d}%
{=}$ denotes equality in distribution of random fields and $SO(3)$ can be
realized as the set of orthonormal $3\times 3$ matrices with unit
determinant.

However, most recent and forthcoming experiments (such as Planck, which was
launched on May 14, 2009, the CLOVER, QUIET and QUAD experiments or the
projected mission CMBPOL) are focussing on a much more elusive and
sophisticated feature, i.e. the so-called polarization of CMB. The physical
significance of the latter is explained for instance in \cite%
{ck05,kks96,selzad}; we do not enter into these motivations here, but we do
stress how the analysis of this feature is expected to provide extremely
rewarding physical information. Just to provide a striking example,
detection of a non-zero angular power spectrum for the so-called $B$-modes
of polarization data (to be defined later) would provide the first
experimental evidence of primordial gravitational waves; this would result
in an impressive window into the General Relativity picture of the
primordial Big Bang dynamics and as such it is certainly one of the most
interesting perspectives of current physical research. Polarization is also
crucial in the understanding of the so-called reionization optical depth,
for which very little information is available from temperature data, see %
\cite{ghmkp} for more discussion on details.

Here, however, we shall not go deeper into these physical perspectives, as
we prefer to focus instead on the new mathematical ideas which are forced in
by the analysis of these datasets. A rigorous understanding requires some
technicalities which are postponed to the next Section; however we hope to
convey the general idea as follows. We can imagine that experiments
recording CMB radiation are measuring on each direction $p\in \mathbb{S}^{2}$
a random ellipse living on $T_{p}\mathbb{S}^{2},$ the tangent plane at that
point. The ``magnitude'' of this ellipse ($=c^{2}=a^{2}+b^{2}$ in standard
ellipse notation), which is a standard random variable, corresponds to
temperature data, on which the mathematical statistics literature has so far
concentrated. The other identifying features of this ellipse (elongation and
orientation) are collected in polarization data, which can be thought of as
a random field taking values in a space of algebraic curves. In more formal
terms (to be explained later), this can be summarized by saying that we
shall be concerned with random sections of fibre bundles over the sphere;
from a more group-theoretic point of view, we shall show that polarization
random fields are related to so-called spin-weighted representations of the
group of rotations $SO(3)$. A further mathematical interpretation, which is
entirely equivalent but shall not be pursued here, is to view these data as
realizations of random matrix fields (see again \cite{schwartzman}). Quite
interestingly, there are other, unrelated situations in physics where the
mathematical and statistical formalism turns out to be identical. In
particular gravitational lensing data, which have currently drawn much
interest in Astrophysics and will certainly make up a core issue for
research in the next two decades, can be shown to have the same (spin 2, see
below) mathematical structure, see for instance (\cite{bridle}). More
generally, similar issues may arise when dealing with random deformations of
shapes, as dealt with for instance by (\cite{anderes}).

The construction of a wavelet system for spin functions was first addressed
in \cite{gelmar}; the idea in that paper is to develop the needlet approach
of \cite{npw1,npw2} and \cite{gm1,gm2,gm3} to this new, broader geometrical
setting, and investigate the stochastic properties of the resulting spin
needlet coefficients, thus generalizing results from \cite{bkmpAoS,bkmpBer}.
A wide range of possible applications to the analysis of polarization data
is discussed in \cite{ghmkp}. Here, we shall focus in particular on the
possibility of using spin needlets for angular power spectrum estimation for
spin fields, an idea that for the scalar case was suggested by \cite{bkmpAoS}%
; in \cite{pbm06}, needlets were used for the estimation of cross-angular
power spectra of CMB and Large Scale Structure data, in \cite{fg08}, \cite%
{fay08} the estimator was considered for CMB temperature data in the
presence of faint noise and gaps, while in \cite{pietrobon1} the procedure
was implemented on disjoint subsets of the sphere as a probe of asymmetries
in CMB radiation.

The plan of this paper is as follows: in Section 2 we present the
motivations for our analysis, i.e. some minimal physical background on
polarization. In Section 3 and 4 we introduce the geometrical formalism on
spin line bundles and spin needlets, respectively, and we define spin random
fields. Sections 5, 6 and 7 are devoted to the spin needlets spectral
estimator and the derivation of its asymptotic properties in the presence of
missing observations and noise, including related statistical tests for bias
and asymmetries. Throughout this paper, given two positive sequences $%
\left\{ a_{j}\right\} ,\left\{ b_{j}\right\} $ we shall write $a_{j}\approx
b_{j}$ if there exist positive constants $c_{1},c_{2}$ such that $%
c_{1}a_{j}\leq b_{j}\leq c_{2}a_{j}$ for all $j\geq 1.$

\section{Motivations}

The classical theory of electromagnetic radiation entails a characterization
in terms of the so-called Stokes' parameters $Q$ and $U$, which are defined
as follows. An electromagnetic wave propagating in the $z$ direction has
components
\begin{equation}
E_{x}(z,t)=E_{0x}\cos (\tau +\delta _{x})\text{ , }E_{y}(z,t)=E_{0y}\cos
(\tau +\delta _{y})\text{ ,}  \label{light}
\end{equation}%
where $\tau :=\omega t-kz$ is the so-called propagator and $\nu =2\pi \omega
/k$ is the frequency of the wave. (\ref{light}) can be viewed as the
parametric equations of an ellipse which is the projection of the incoming
radiation on the plane perpendicular to the direction of motion. Indeed,
some elementary algebra yields
\begin{equation*}
\frac{E_{x}^{2}(z,t)}{E_{0x}^{2}}+\frac{E_{y}^{2}(z,t)}{E_{0y}^{2}}-2\frac{%
E_{x}(z,t)}{E_{0x}}\frac{E_{y}(z,t)}{E_{0y}}\cos \delta =\sin ^{2}\delta
\text{ , }\delta :=\delta _{y}-\delta _{x}\text{ .}
\end{equation*}%
The magnitude of the ellipse (i.e., the sums of the squares of its semimajor
and semiminor axes) is given by%
\begin{equation*}
T=E_{0x}^{2}+E_{0y}^{2}\text{ ;}
\end{equation*}%
$T$ has the nature of a scalar quantity, that is to say, it is readily seen
to be invariant under rotation of the coordinate axis $x$ and $y.$ It can
hence be viewed as an intrinsic quantity measuring the total intensity of
radiation; from the physical point of view, this is exactly the nature of
CMB\ temperature observations which have been the focus of so much research
over the last decade. It should be noted that, despite the non-negativity
constraint, in the physical literature on CMB experiments\ $T$ is usually
taken to be Gaussian around its mean, in excellent agreement with
observations. This apparent paradox is explained by the fact that the
variance of $T$ is several orders of magnitude smaller than its mean, so the
Gaussian approximation is justifiable.

The characterization of the polarization ellipse is completed by introducing
Stokes' parameters $Q$ and $U$, which are defined as
\begin{equation}
Q=E_{0x}^{2}-E_{0y}^{2}\text{ , }U=2E_{0x}E_{0y}\cos \delta \text{ .}
\label{stokes}
\end{equation}%
To provide a flavour of their geometrical meaning, we recall from elementary
geometry that the parametric equations of a circle are obtained from (\ref%
{light}) in the special case $E_{0x}=E_{0y},$ $\delta _{x}=\delta _{y}+\pi
/2,$ whence the circle corresponds to $Q=U=0.$ On the other hand, it is not
difficult to see that a segment aligned on the $x$ axis is characterized by $%
Q=T,$ a segment aligned on the $y$ axis by $Q=-T,$ for a segment on the line
$y=\pm x$ we have $\delta _{x}-\delta _{y}=0,\pi ,$ and hence $Q=0,$ $U=\pm
T $, respectively. The key feature to note, however, is the following: while
$T $ does not depend on any choice of coordinates, this is not the case for $%
Q$ and $U,$ i.e. the latter are not geometrically intrinsic quantities.
However, as these parameters identify an ellipse, it is natural to expect
that they will be invariant under rotations by $180^{\circ }$ degrees and
multiples thereof. This is the first step in understanding the introduction
of spin random fields below.

Indeed, it is convenient to identify $\mathbb{R}^{2}$ with the complex plane
$\mathbb{C}$ by focussing on $w=x+iy;$ a change of coordinates corresponding
to a rotation $\gamma $ can then be expressed as $w^{\prime }=\exp (i\gamma
)w,$ and some elementary algebra shows that the induced transform on $(Q,U)$
can be written as%
\begin{equation*}
\left(
\begin{array}{c}
Q^{\prime } \\
U^{\prime }%
\end{array}%
\right) =\left(
\begin{array}{cc}
\cos 2\gamma & \sin 2\gamma \\
-\sin 2\gamma & \cos 2\gamma%
\end{array}%
\right) \left(
\begin{array}{c}
Q^{\prime } \\
U^{\prime }%
\end{array}%
\right) \text{ ,}
\end{equation*}%
or more compactly
\begin{equation}
Q^{\prime }+iU^{\prime }=\exp (i2\gamma )(Q+iU)\text{ .}  \label{spinqu}
\end{equation}%
In the physicists' terminology, (\ref{spinqu}) identifies the Stokes'
parameters as spin 2 objects, that is, a rotation by an angle $\gamma $
changes their value by $\exp (i2\gamma ).$ As mentioned before, this can be
intuitively visualized by focussing on an ellipse, which is clearly
invariant by rotations of $180^{%
%TCIMACRO{\U{b0}}%
%BeginExpansion
{{}^\circ}%
%EndExpansion
}.$ To compare with other situations, standard (scalar) random fields are
clearly invariant (or better covariant) with respect to the choice of any
coordinate axes in the local tangent plane, and as such they are spin zero
fields; a vector field is spin 1, while we can envisage random fields taking
values in higher order algebraic curves and thus having any integer spin $%
s\geq 2.$

As mentioned earlier, it is very important to notice that polarization is
not the only possible motivation for the analysis of spin random fields. For
instance, an identical formalism is derived when dealing with gravitational
lensing, i.e. the deformation of images induced by gravity according to
Einstein's laws. Gravitational lensing is now the object of very detailed
experimental studies, which have led to huge challenges on the most
appropriate statistical methods to be adopted (see for instance \cite{bridle}%
). We defer to future work a discussion on the statistical procedures which
are made possible by the application of spin needlets to lensing data.

\section{Geometric Background}

In this Section, we will provide a more rigorous background on spin
functions. Despite the fact that our motivating applications are limited to
the case $s=2,$ we will discuss here the case of a general integer $s\in
\mathbb{Z}$, which does not entail any extra difficulty.

A more rigorous point of view requires some background in Differential
Geometry, for which we refer for instance to \cite{bishop} and \cite{adtay}.
The construction of spin functions is discussed in more detail by \cite%
{gelmar}, which builds upon a well-established physical literature described
for instance in \cite{newman}, \cite{selzad}, \cite{ck05}.

Say $p\in \mathbb{S}^{2}$. We recall first the tangent plane $T_{p}\mathbb{S}%
^{2},$ which is defined as usual as the linear space generated by the
collection of tangent vectors at $p.$ To proceed further to spin random
fields, we need to recall from Geometry the notion of a fibre bundle. The
latter consists of the family $(E,B,\pi ,F)$, where $E$, $B$, and $F$ are
topological spaces and $\pi :E\rightarrow B$ is a continuous surjection
satisfying a local triviality condition outlined below. The space $B$ is
called the base space of the bundle, $E$ the total space, and $F$ the fibre;
the map $\pi $ is called the projection map (or bundle projection). In our
case, the base space is simply the unit sphere $B=\mathbb{S}^{2}$; it is
tempting to view the fibres as ellipses (or vectors, for $s=1,$ or more
general algebraic curves, for $s\geq 3)$ lying in $T_{p}\mathbb{S}^{2}$,
however one must bear in mind that to characterize the ellipse we would need
to focus jointly on $(T,Q,U),$ while our analysis below is restricted to the
Stokes' parameters $Q$ and $U.$

The basic intuition behind fibre bundles is that they behave locally as
simple Cartesian products $B\times F$. The former intuition is implemented
by requiring that for all $p\in \mathbb{S}^{2}$ there exist a neighbourhood $%
U=U(p)$ such that $\pi ^{-1}(U)$ is homeomorphic to $U\times F$, in such a
way that $\pi $ carries over to the projection onto the first factor. In
particular, the following diagram should commute:%
\begin{equation*}
\begin{array}{ccc}
\pi ^{-1}(U) & \overset{\phi }{\longmapsto } & U\times F \\
\pi \downarrow & \underset{proj}{\swarrow } &  \\
U &  &
\end{array}%
\text{ },
\end{equation*}%
where $\phi $ is a homeomorphism and $proj$ is the natural projection. The
set $\pi ^{-1}(x)$ is homeomorphic to $F$ and is called the fibre over $x.$
The fibre bundles we shall consider are smooth, that is, $E,B,$ and $F$ are
required to be smooth manifolds and all the projections above are required
to be smooth maps.

In our case, we shall be dealing with a complex line bundle which is
uniquely identified by fixing transition functions to express the
transformation laws under changes of coordinates. Following (\cite{gelmar})
(see also \cite{goldb,newman}), we define $U_{I}:=\mathbb{S}^{2}\setminus
\left\{ N,S\right\} $ to be the chart covering the sphere with the exception
of the North and South Poles, with the usual coordinates $(\vartheta
,\varphi ).$ We define also the rotated charts $U_{R}=RU_{I};$ in this new
charts, we will use the natural coordinates $(\vartheta _{R},\varphi _{R}).$
At each point $p$ of $U_{R}$, we take as a reference direction in the
tangent plane $T_{p}S^{2}$, the tangent vector $\partial /\partial \varphi
_{R}$, (which points in the direction of increasing $\varphi _{R}$ and is
tangent to a circle $\theta _{R}=$ constant). Again as in \cite{gelmar}, we
let let $\psi _{pR_{2}R_{1}}$ be the (oriented) angle from $\partial
/\partial \varphi _{R_{1}}$ to $\partial /\partial \varphi _{R_{2}}$ (for a
careful discussion of which is the oriented angle, see \cite{gelmar}); this
angle is independent of any choice of coordinates. We define a complex line
bundle on $S^{2}$ by letting $\exp (is\psi _{pR_{2}R_{1}})$ be the
transition function from the chart $U_{R_{1}}$ to $U_{R_{2}}$. A smooth spin
function $f$ is a smooth section of this line bundle. $f$ may simply be
thought of as a collection of complex-valued smooth functions $(f_{R})_{R\in
SO(3)}$, with $f_{R}$ defined and smooth on $U_{R}$,, such that for all $%
R_{1},R_{2}\in SO(3)$, we have
\begin{equation*}
f_{R_{2}}(p)=\exp (is\psi _{pR_{2}R_{1}})f_{R_{1}}(p)
\end{equation*}%
for all $p$ in the intersection of $U_{R_{1}}$ and $U_{R_{2}}$.

An alternative, group theoretic point of view can be motivated as follows.
Consider the group of rotations $SO(3);$ it is a well-known that, by
elementary geometry, each element $g$ can be expressed as
\begin{equation}
g=R_{z}(\alpha )R_{x}(\beta )R_{z}(\gamma )\text{ , }0\leq \alpha \leq \pi
\text{ , }0\leq \beta ,\gamma \leq 2\pi \text{ ,}  \label{euler}
\end{equation}%
where $R_{z}(.)$ and $R_{x}(.)$ represent rotations around the $z$ and $x$
axis, respectively; in words, (\ref{euler}) is stating that each rotation
can be realized by rotating first by an angle $\gamma $ around the $z$ axis,
then by an angle $\beta $ around the $x$ axis, then again by an angle $%
\alpha $ around the $z$ axis. We denote as usual by $\left\{
D^{l}(.)\right\} _{l=0,1,2,...}$ the Wigner family of irreducible matrix
representations for $SO(3);$ in terms of the Euler angles, the elements of
these matrices can be expressed as
\begin{equation*}
D_{m_{1}m_{2}}^{l}(g)=\exp (-im_{1}\alpha )d_{m_{1}m_{2}}^{l}(\beta )\exp
(-im_{2}\gamma )\text{ .}
\end{equation*}%
Note that $\overline{D_{m_{1}m_{2}}^{l}}(g)=(-1)^{m_{1}-m_{2}}\overline{%
D_{-m_{1},-m_{2}}^{l}}(g).$ Standard results from group representation
theory (\cite{faraut,vmk,VIK}) yield
\begin{equation*}
\sum_{m_{2}}D_{m_{1}m_{2}}^{l}(g)\overline{D_{m_{1}^{\prime
}m_{2}}^{l^{\prime }}}(g) = \delta _{l}^{l^{\prime }}\delta
_{m_{1}}^{m_{1}^{\prime }}\text{ ,}
\end{equation*}
and
\begin{equation*}
  \int_{SO(3)}D_{m_{1}m_{2}}^{l}(g)\overline{D_{m_{1}^{\prime
}m_{2}^{\prime }}^{l^{\prime }}}(g)dg = \frac{8\pi ^{2}}{2l+1}\delta
_{l}^{l^{\prime }}\delta _{m_{1}}^{m_{1}^{\prime }}\delta
_{m_{2}}^{m_{2}^{\prime }}\text{ },
\end{equation*}
$dg$ denoting the standard uniform (Haar) metric on $SO(3).$ The elements of
$\left\{ D^{l}(.)\right\} _{l=0,1,2,...}$ thus make up an orthogonal system
which is also complete, i.e., it is a consequence of the Peter-Weyl theorem (%
\cite{faraut}) that all square integrable functions on $SO(3)$ can be
expanded, in the mean square sense, as
\begin{equation*}
f(g)=\sum_{l}\sum_{m_{1}m_{2}}\frac{2l+1}{8\pi ^{2}}%
b_{m_{1}m_{2}}^{l}D_{m_{1}m_{2}}^{l}(g)\text{ ,}
\end{equation*}%
where the coefficients $\left\{ b_{m_{1}m_{2}}^{l}\right\} $ can be
recovered from the inverse Fourier transform%
\begin{equation*}
b_{m_{1}m_{2}}^{l}=\int_{SO(3)}f(g)\overline{D_{m_{1}m_{2}}^{l}(g)}dg\text{ .%
}
\end{equation*}%
By elementary geometry, we can view the unit sphere as the quotient space $%
\mathbb{S}^{2}=SO(3)/SO(2)$ and the functions on the sphere as those which
are constants with respect to the third Euler angle $\gamma ,$ i.e. $%
f(\alpha ,\beta ,\gamma )=f(\alpha ,\beta ,\gamma ^{\prime })$ for all $%
\gamma ,\gamma ^{\prime }.$ It follows that
$$
\displaylines{
\int_{SO(3)}f(g)\overline{D_{m_{1}m_{2}}^{l}(g)}\,dg
\cr
=(-1)^{m_{1}-m_{2}}\int_{0}^{2\pi }\int_{0}^{\pi }\int_{0}^{2\pi
}f(g)\exp (im_{1}\alpha )d_{-m_{1}-m_{2}}^{l}(\beta )\exp
(im_{2}\gamma )\sin \beta
d\alpha d\beta d\gamma \cr
=\begin{cases}
0&\mbox{for }m_{2}\neq 0 \cr
2\pi b_{m_{1}0}^{l}&\mbox{otherwise}\cr
\end{cases}
}
$$
%\begin{eqnarray*}
%&\displaystyle\int_{SO(3)}f(g)\overline{D_{m_{1}m_{2}}^{l}(g)}\,dg \\
%&\displaystyle=(-1)^{m_{1}-m_{2}}\int_{0}^{2\pi }\int_{0}^{\pi
%}\int_{0}^{2\pi }f(g)\exp (im_{1}\alpha )d_{-m_{1}-m_{2}}^{l}(\beta
%)\exp (im_{2}\gamma )\sin \beta
%d\alpha d\beta d\gamma \\
%\displaystyle&=\left\{
%\begin{array}{c}
%0\text{ for }m_{2}\neq 0 \\
%2\pi b_{m_{1}0}^{l}\text{ otherwise}%
%\end{array}%
%\right.
%\end{eqnarray*}%
In view of the well-known identity%
\begin{equation*}
\begin{array}{c}
\displaystyle Y_{lm}(\beta ,\alpha ) =\sqrt{\frac{2l+1}{4\pi }}d_{m0}^{l}(\beta
)e^{im\alpha } \\
\displaystyle =(-1)^{m}\sqrt{\frac{2l+1}{4\pi }}d_{-m0}^{l}(\beta )e^{im\alpha }=(-1)^{m}%
\sqrt{\frac{2l+1}{4\pi }}D_{-m,0}^{l}(\alpha ,\beta ,\gamma )
\end{array}
\end{equation*}
%
%\begin{eqnarray*}
%Y_{lm}(\beta ,\alpha ) &=&\sqrt{\frac{2l+1}{4\pi }}d_{m0}^{l}(\beta
%)e^{im\alpha } \\
%&=&(-1)^{m}\sqrt{\frac{2l+1}{4\pi }}d_{-m0}^{l}(\beta )e^{im\alpha }=(-1)^{m}%
%\sqrt{\frac{2l+1}{4\pi }}D_{-m,0}^{l}(\alpha ,\beta ,\gamma )\text{ , }
%\end{eqnarray*}%
(where we have used $d_{mn}^{l}(\beta )=(-1)^{m-n}d_{-m,-n}^{l}(\beta ),$
see \cite{vmk}$,$ equation 4.4.1), we immediately obtain the expansion of
functions on the sphere into spherical harmonics, i.e.
\begin{equation*}
f(p)=\sum_{l}\sum_{m}\frac{2l+1}{4\pi }b_{m0}^{l}D_{m0}^{l}(p)=%
\sum_{lm}a_{lm}Y_{lm}(p)\text{ , }a_{lm}=\sqrt{\frac{2l+1}{4\pi }}b_{-m0}^{l}%
\text{ .}
\end{equation*}%
We can hence loosely say that standard scalar functions on the sphere ``live
in the space generated by the column $s=0$ of the Wigner's $D$ matrices of
irreducible representations'', see also \cite{mpjmva}. Now from the
Peter-Weyl Theorem we know that each of the columns $s=-l,...,l$ spans a
space of irreducible representations, and these spaces are mutually
orthogonal; it is then a natural, naive question to ask what is the physical
significance of these further spaces. It turns out that these are strictly
related to spin functions, indeed we can expand the fibre bundle of spin $s$
functions as%
\begin{equation}
f_{s}(\vartheta ,\varphi )=\left. \sum_{l}\sum_{m}\frac{2l+1}{4\pi }%
b_{ms}^{l}D_{ms}^{l}(\varphi ,\vartheta ,\gamma )\right| _{\gamma =0}.
\label{specwig}
\end{equation}%
Spin $s$ functions can then be related to the so-called spin weighted
representations of $SO(3),$ see for instance \cite{brocker}. Now by standard
group representation properties we have that%
\begin{equation*}
   \begin{array}{c}
\displaystyle f_{s}((R_{z}(\gamma )p) =\sum_{l}\sum_{m}\frac{2l+1}{4\pi }%
b_{ms}^{l}D^{l}(R_{z}(\gamma ))D_{ms}^{l}(\varphi ,\vartheta ,\gamma ) \\
\displaystyle =\sum_{l}\sum_{m}\frac{2l+1}{4\pi }b_{lms}\exp (is\gamma
)D_{ms}^{l}(\varphi ,\vartheta ,\gamma )\\
\displaystyle =\exp (is\gamma )f_{s}(p)\text{ ,}
   \end{array}
\end{equation*}
as expected.

The analogy with the scalar case can actually be pursued further than that.
It is well-known that the elements $D_{m0}^{l},$ $m=-l,...,l$ of the
Wigner's $D$ matrices are proportional to the spherical harmonics $Y_{lm},$
i.e. the eigenfunctions of the spherical Laplacian operator \newline
$\Delta _{\mathbb{S}^{2}}Y_{lm}=-l(l+1)Y_{lm}.$ It turns out that this
equivalence holds in much greater generality and for all integer $s$ and $%
l\geq s$ there exist a differential operator $\eth \overline{\eth }$ such
that $-\eth \overline{\eth }D_{ms}^{l}=e_{ls}D_{ms}^{l},$ where $\left\{
e_{ls}\right\} _{l=s,s+1,}=\left\{ (l-s)(l+s+1)\right\} _{s,s+1,...}$ is the
associated sequence of eigenvalues (note that for $s=0$ we are back to the
usual expressions for the scalar case, as expected)$.$ The operators $\eth ,%
\overline{\eth }$ are defined as follows, in terms of their action on any
spin $s$ function $f_{s}(.),$%

\begin{align}
\eth f_{s}(\vartheta ,\varphi ) =-\left( \sin \vartheta \right) ^{s}\left[
\frac{\partial }{\partial \vartheta }+\frac{i}{\sin \vartheta }\frac{%
\partial }{\partial \varphi }\right] \left( \sin \vartheta \right)
^{-s}f_{s}(\vartheta ,\varphi )\text{ },  \label{edth}\\
\overline{\eth }f_{s}(\vartheta ,\varphi ) =-\left( \sin \vartheta \right)
^{-s}\left[ \frac{\partial }{\partial \vartheta }-\frac{i}{\sin \vartheta }%
\frac{\partial }{\partial \varphi }\right] \left( \sin \vartheta \right)
^{s}f_{s}(\vartheta ,\varphi )\text{ }.  \label{bedth}
\end{align}

In (\ref{edth}) one should more rigorously write $(\eth f_s)_I$ on the left
side and $(f_s)_I$ on the right side. In fact, if on the right side of (\ref%
{edth}) we replace $(\vartheta,\varphi)$ by $(\vartheta_R,\varphi_R)$ and $%
f_s$ by $(f_s)_R$, the result is in fact $(\eth
f_s)_R(\vartheta_R,\varphi_R) $ (see \cite{gelmar}); similarly in (\ref%
{bedth}).

The spin $s$ spherical harmonics can then be identified as

\begin{equation}\label{sphwig}
\begin{array}c
Y_{lms}(\vartheta ,\varphi )=(-1)^{m}\sqrt{\frac{2l+1}{4\pi }}%
D_{-ms}^{l}(\varphi ,\vartheta ,-\psi )\exp (-is\psi ) \\
=(-1)^{m}\sqrt{\frac{2l+1}{4\pi }}\exp (im\varphi )d_{-ms}^{l}(\vartheta )%
\text{ ;}
\end{array}
\end{equation}%
again, the previous expression should be understood as $Y_{lms;I}(\vartheta
,\varphi ),$ i.e. spin spherical harmonics are clearly affected by
coordinate transformation, but we drop the reference to the choice of chart
for ease of notation whenever this can be done without the risk of
confusion. The spin spherical harmonics can be shown to satisfy%
\begin{equation}\label{spinh}
Y_{lm,s+1} =\left[ \left( l-s\right) \left( l+s+1\right) \right]
^{-1/2}\eth Y_{lm,s}\text{ },
\end{equation}

\begin{equation}\label{spinh2}
Y_{lm,s-1} =-\left[ \left( l+s\right) \left( l-s+1\right) \right] ^{-1/2}%
\overline{\eth }Y_{lm,s}\text{ },
\end{equation}%
which motivates the name of spin raising and spin lowering operators for $%
\eth ,\overline{\eth }$. Iterating, it can be shown also that (see \cite%
{gelmar})%
\begin{equation*}
Y_{lms} =\left\{ \frac{(l-s)!}{(l+s)!}\right\} ^{1/2}(\eth )^{s}Y_{lm}%
\text{ , for }s>0\text{ ,}
\end{equation*}
\begin{equation*}
Y_{lms} =\left\{ \frac{(l+s)!}{(l-s)!}\right\} ^{1/2}(-\overline{\eth }%
)^{-s}Y_{lm}\text{ , for }s<0\text{ .}
\end{equation*}%
Further properties of the spin spherical harmonics follow easily from their
proportionality to elements of Wigner's $D$ matrices; indeed we have
(orthonormality)%
\begin{equation*}
\int_{\mathbb{S}^{2}}Y_{lms}(p)\overline{Y_{l^{\prime }m^{\prime }s}}%
(p)dp=\int_{0}^{2\pi }\int_{0}^{\pi }Y_{lms}(\vartheta ,\varphi )\overline{%
Y_{l^{\prime }m^{\prime }}(\vartheta ,\varphi )}\sin \vartheta d\vartheta
d\varphi =\delta _{l}^{l^{\prime }}\delta _{m}^{m^{\prime }}\text{ .}
\end{equation*}%
Viewing spin-spherical harmonics as functions on the group $SO(3)$ (i.e.
identifying $p=(\vartheta ,\varphi )$ as the corresponding rotation by means
of Euler angles), using (\ref{sphwig}) and the group addition properties we
obtain easily, for $p,p^{\prime}\in \mathbb{S}^{2}$, that

\begin{equation*}
\sum_{m=-l}^{l}Y_{lms}\left( p\right) \overline{Y_{lms}\left( p^{\prime
}\right) } =\frac{2l+1}{4\pi }\sum_{m}D_{-ms}^{l}(\varphi ,\vartheta ,0)%
\overline{D_{-ms}^{l}}(\varphi ^{\prime },\vartheta ^{\prime },0)
\end{equation*}
\begin{equation*}
=\frac{2l+1}{4\pi }D_{-ss}^{l}(\psi (p,p^{\prime }))\text{ ,}
\end{equation*}
where $\psi (p,p^{\prime })$ denotes the composition of the two rotations
(explicit formulae can be found in \cite{vmk}). In the special case $%
p=p^{\prime }$ and $R=R^{\prime },$ we have immediately
\begin{equation}
\sum_{m=-l}^{l}Y_{lms}\left( p\right) \overline{Y_{lms}\left( p\right) }=%
\frac{2l+1}{4\pi }\text{ ,}  \label{addup}
\end{equation}%
see also \cite{gelmar} for an alternative proof.

By combining (\ref{specwig}) and (\ref{sphwig}) the spectral representation
of spin functions is derived:%
\begin{equation}
f_{s}(\vartheta ,\varphi )=\sum_{l}\sum_{m}a_{l;ms}Y_{l;ms}(\vartheta
,\varphi )\text{ .}  \label{specspin}
\end{equation}%
From (\ref{specspin}), a further, extremely important characterization of
spin functions was first introduced by \cite{newman}, see also and \cite%
{gelmar} for a more mathematically oriented treatment. In particular, it can
be shown that there exists a scalar complex-valued function $g(\vartheta
,\varphi )={\rm Re}\left\{ g\right\} (\vartheta ,\varphi )+i{\rm Im}%
\left\{ g\right\} (\vartheta ,\varphi )$, such that, such that%
\begin{equation*}
f_{s}(\vartheta ,\varphi )=f_{E}(\vartheta ,\varphi )+if_{B}(\vartheta
,\varphi )
\end{equation*}%
\begin{equation}
=\sum_{lm}a_{lm;E}Y_{lms}(\vartheta ,\varphi
)+i\sum_{lm}a_{lm;B}Y_{lms}(\vartheta ,\varphi )\text{ ,}  \label{charspin}
\end{equation}%
where
\begin{equation*}
f_{E}(\vartheta ,\varphi )=(\eth )^{s}{\rm Re}\left\{ g\right\}
(\vartheta ,\varphi )\text{ , }f_{B}(\vartheta ,\varphi )=(\eth
)^{s}{\rm Im}\left\{ g\right\} \text{ .}
\end{equation*}%
Note that $a_{l;ms}=a_{lm;E}+ia_{lm;B}$, where $a_{lm;E}=\overline{a}%
_{l-m;E},$ $a_{lm;B}=\overline{a}_{l-m;B}.$ It is also readily seen that%

\begin{equation*}
a_{l;ms}+\overline{a_{l;-ms}}
=a_{lm;E}+ia_{lm;B}+a_{lm;E}-ia_{lm;B}=2a_{lm;E}\text{ ,}
\end{equation*}
\begin{equation*}
a_{l;ms}-\overline{a_{l;-ms}}
=a_{lm;E}+ia_{lm;B}-a_{lm;E}+ia_{lm;B}=2ia_{lm;B}\text{ .}
\end{equation*}

In the cosmological literature, $\left\{ a_{lm;E}\right\} $ and $\left\{
a_{lm;B}\right\} $ are labelled the $E$ and $B$ modes (or the electric and
magnetic components) of CMB\ polarization.

\section{Spin Needlets and Spin Random Fields}

We are now in the position to recall the construction of spin needlets, as
provided by \cite{gelmar}. We start by reviewing a few basic facts about
standard (scalar) needlets. Needlets have been defined by \cite{npw1,npw2}
as
\begin{equation}
\psi _{jk}\left( p\right) =\sqrt{\lambda _{jk}}\sum_{l}b\left( \frac{l}{B^{j}%
}\right) \sum_{m=-l}^{l}Y_{lm}\left( p\right) \overline{Y_{lm}}\left( \xi
_{jk}\right) \text{ , }p\in \mathbb{S}^{2},  \label{spinstan}
\end{equation}%
where $\left\{ \xi _{jk},\lambda _{jk}\right\} $ are a set of cubature
points and weights ensuring that%
\begin{equation*}
\sum_{jk}\lambda _{jk}Y_{lm}\left( \xi _{jk}\right) \overline{Y_{l^{\prime
}m^{\prime }}}\left( \xi _{jk}\right) =\int_{\mathbb{S}^{2}}Y_{lm}\left(
p\right) \overline{Y_{l^{\prime }m^{\prime }}}\left( p\right) dp=\delta
_{l}^{l^{\prime }}\delta _{m}^{m^{\prime }}\text{ ,}
\end{equation*}%
$b(.)$ is a compactly supported, $C^{\infty }$ function, and $B>1$ is a
user-chosen ``bandwidth''\ parameter. The general cases of non-compactly
supported functions $b(.)$ and more abstract manifolds than the sphere were
studied by \cite{gm1,gm2,gm3}. The stochastic properties of needlet
coefficients and their use for the analysis of spherical random fields were
first investigated by \cite{bkmpAoS,bkmpBer}, see also \cite{ejslan, mayeli,
spalan, fg08} for further developments. Several applications have already
been provided to CMB data analysis, see for instance \cite%
{pbm06,mpbb08,dela08,fay08,
pietrobon1,pietrobon2,rudjord1,pietrobon3,rudjord2}.

For a fixed $B>1,$ we shall denote by $\{\mathcal{X}_{j}\}_{j=0}^{\infty }$
the nested sequence of cubature points corresponding to the space $\mathcal{K%
}_{[2B^{j+1}]},$ where $[.]$ represents as usual integer part and $\mathcal{K%
}_{L}=\oplus _{l=0}^{L}H_{l}$ is the space spanned by spherical harmonics up
to order $L$. It is known that $\{\mathcal{X}_{j}\}_{j=0}^{\infty }$ can be
taken such that the cubature points for each $j$ are almost uniformly $%
\epsilon _{j}-$distributed with $\epsilon _{j}:=\kappa B^{-j},$ the
coefficients $\{\lambda _{jk}\}$ are such that $cB^{-2j}\leq \lambda
_{jk}\leq c^{\prime }B^{-2j}$, where $c,c^{\prime }$ are finite real
numbers, and $card\left\{ \mathcal{X}_{j}\right\} \approx B^{2j}$. Exact
cubature points can be defined for the spin as for the scalar case, see \cite%
{bkmp09} for details; for practical CMB data analysis, these cubature points
can be identified with the centre pixels provided \cite{healpix}, with only
a minor approximation.

Spin needlets are then defined as (see \cite{gelmar})
\begin{equation}
\psi _{jk;s}\left( p\right) =\sqrt{\lambda _{jk}}\sum_{l}b\left( \frac{\sqrt{%
e_{ls}}}{B^{j}}\right) \sum_{m=-l}^{l}Y_{l;ms}\left( p\right) \overline{%
Y_{l;ms}}\left( \xi _{jk}\right) \text{ .}  \label{spinlet}
\end{equation}%
As before, $\left\{ \lambda _{jk},\xi _{jk}\right\} $ are cubature points
and weights, $b\left( \cdot \right) \in C^{\infty }$ is nonnegative$,$ and
has a compact support in $\left[ 1/B,B\right] .$ The expression (\ref%
{spinlet}) bears an obvious resemblance with (\ref{spinstan}), but it is
also important to point out some crucial differences. Firstly, we note that
the square root of the eigenvalues $\sqrt{e_{ls}}$ has replaced the previous
$l.$ This formulation is instrumental for the derivation of the main
properties of spin needlets by means of differential arguments in (\cite%
{gelmar}); we stress, however, that this is actually a minor difference, as
all our results are asymptotic and of course
\begin{equation*}
\lim_{l\rightarrow \infty }\frac{\sqrt{e_{ls}}}{l}=\lim_{l\rightarrow \infty
}\frac{\sqrt{(l-s)(l+s+1)}}{l}=1\text{ for all fixed }s\text{ .}
\end{equation*}%
A much more important feature is as follows: (\ref{spinlet}) cannot be
viewed as a well-defined scalar or spin function, because $Y_{l;ms}\left(
p\right) ,\overline{Y_{l;ms}}\left( \xi _{jk}\right) $ are spin($s$ and $-s)$
functions defined on different point of $\mathbb{S}^{2},$ and as such they
cannot be multiplied in any meaningful way (their product depends on the
local choice of coordinates). Hence, (\ref{spinlet}) should be written more
rigorously as
\begin{equation*}
\psi _{jk;s}\left( p\right) =\sqrt{\lambda _{jk}}\sum_{l}b\left( \frac{%
\sqrt{e_{ls}}}{B^{j}}\right) \sum_{m=-l}^{l}\left\{ Y_{l;ms}\left( p\right)
\otimes \overline{Y_{l;ms}}\left( \xi _{jk}\right) \right\} \text{ ,}
\end{equation*}
\begin{equation*}
\overline{\psi _{jk;s}}\left( p\right) =\sqrt{\lambda _{jk}}%
\sum_{l}b\left( \frac{\sqrt{e_{ls}}}{B^{j}}\right) \sum_{m=-l}^{l}\left\{
\overline{Y_{l;ms}}\left( p\right) \otimes Y_{l;ms}\left( \xi _{jk}\right)
\right\} \text{ ,}
\end{equation*}
where we denoted by $\otimes $ the tensor product of spin functions; spin
needlets can the be viewed as spin $\left\{ -s,s\right\} $ operators
(written $T_{-s,s})$, which act on a space of spin $s$ functions square
integrable functions to produce a sequence of spin $s$ square-summable
coefficients, i.e. $T_{-s,s}:L_{s}^{2}\rightarrow \ell _{s}^{2}.$ This
action is actually an isometry, as a consequence of the tight frame
property, see \cite{bkmp09} and \cite{gm4}.

For any spin $s$ function $f_{s},$ the spin needlet transform is defined by%
\begin{equation*}
\int_{\mathbb{S}^{2}}f_{s}(p)\overline{\psi _{jk;s}}(p)dp=\beta _{jk;s}\text{
,}
\end{equation*}%
and the same inversion property holds as for standard needlets, i.e.
\begin{equation*}
f_{s}(p)=\sum_{jk}\beta _{jk;s}\psi _{jk;s}(p)\text{ ,}
\end{equation*}%
the equality holding in the $L^{2}$ sense. The coefficients of spin needlets
can be written explicitly as%
\begin{equation}
\beta _{jk;s}=\int_{\mathbb{S}^{2}}f_{s}(p)\overline{\psi _{jk;2}}(p)dp=%
\sqrt{\lambda _{jk}}\sum_{l}b\left( \frac{\sqrt{e_{ls}}}{B^{j}}\right)
\sum_{m=-l}^{l}a_{l;ms}Y_{l;ms}\left( \xi _{jk}\right) \text{ }.
\label{coef}
\end{equation}

\begin{remark}
To illustrate the meaning of these projection operations, and using a
notation closer to the physical literature, we could view spin $s$
quantities as ``bra'' entities, i.e. write $\left\langle T(p)\right.
,\left\langle \beta _{jk;s}\right. ,$ and spin $-s$ as ``ket'' quantities,
i.e. write for instance $\left. \overline{Y_{l;ms}}\left( p\right)
\right\rangle .$ Then we would obtain
\begin{equation*}
\int_{\mathbb{S}^{2}}f_{s}(p)\overline{\psi _{jk;2}}(p)dp =\sqrt{\lambda
_{jk}}\sum_{l}b\left( \frac{\sqrt{e_{ls}}}{B^{j}}\right)
\sum_{m=-l}^{l}\int_{\mathbb{S}^{2}}\left\langle f_{s}(p)\right. ,\left.
\overline{Y_{l;ms}}\left( p\right) \right\rangle \left\langle Y_{l;ms}\left(
\xi _{jk}\right) \right. dp
\end{equation*}
\begin{equation*}
=\sqrt{\lambda _{jk}}\sum_{l}b\left( \frac{\sqrt{e_{ls}}}{B^{j}}\right)
\sum_{m=-l}^{l}a_{l;ms}\left\langle Y_{l;ms}\left( \xi _{jk}\right) \right.
\text{ ,}
\end{equation*}%
which is a well-defined spin quantities, as the inner product $\left\langle
f_{s}(p),\overline{Y_{l;ms}}\left( p\right) \right\rangle $ yields a
well-defined, complex-valued scalar. However we shall not use this ``Dirac''
notation later in this paper, as we hope the meaning of our manipulations
will remain clear by themselves.
\end{remark}

The absolute value of spin needlets is indeed a well-defined scalar
function, and this allows to discuss localization properties. In this
framework, the main result is established in \cite{gelmar}, where it is
shown that for any $M\in \mathbb{N}$ there exists a constant $c_{M}>0$ s.t.,
for every $\xi \in \mathbb{S}^{2}$:
\begin{equation}
\left| \psi _{jk;s}(\xi )\right| \leq \frac{c_{M}B^{j}}{(1+B^{j}\arccos
(\langle \xi _{jk},\xi \rangle ))^{M}}\text{ uniformly in }(j,k)\text{ },
\label{3}
\end{equation}%
i.e. the tails decay quasi-exponentially.

We are now able to focus on the core of this paper, which is related to the
analysis of spin random fields. As mentioned in the previous discussion, we
have in mind circumstances where stochastic analysis must be developed on
polarization random fields $\left\{ Q\pm iU\right\} ,$ which are spin $\pm 2$
random functions.

Hence we shall now assume we deal with random isotropic spin functions $f_s$%
, by which we mean that there exist a probability space $(\Omega ,\Im ,P)$,
such that for all choices of charts $U_{R}$, the ordinary random function $%
(f_{s})_R$, defined on $\Omega \times \mathbb{S}^{2}$, is jointly $\Im
\times \mathcal{B}(U_R)$ measurable, where $\mathcal{B}(U_R)$ denotes the
Borel sigma-algebra on $U_R.$ In particular, for the spin $2$ random
function $(Q+iU)(p)$ as for the scalar case, the following representation
holds, in the mean square sense (\cite{gelmar})%
\begin{equation*}
\left\{ Q+iU\right\} =\sum_{lm}a_{lm;2}Y_{l;m2}\text{ , }
\end{equation*}%
i.e.
\begin{equation*}
\lim_{L\rightarrow \infty }E\int_{\mathbb{S}^{2}}\left| \left\{ Q+iU\right\}
(p)-\sum_{l=1}^{L}\sum_{m=-l}^{l}a_{lm;2}Y_{l;m2}(p)\right| ^{2}dp=0\text{ .}
\end{equation*}%
Note that the quantity on the left-hand side is a well-defined scalar, for
all $L.$ The sequence $\left\{ a_{lm2}=a_{lm;E}+ia_{lm;B}\right\} $ is
complex-valued and is such that, for all $l_{1},l_{2},m_{1},m_{2}$ ,%
\begin{equation*}
Ea_{l_{1}m_{1};E}a_{l_{2}m_{2};E}=Ea_{l_{1}m_{1};B}a_{l_{2}m_{2};E}=Ea_{l_{1}m_{1};E}a_{l_{2}m_{2};B}=Ea_{l_{1}m_{1};E}%
\overline{a_{l_{2}m_{2};B}}=0\text{ , }
\end{equation*}%
and
\begin{equation*}
Ea_{lm;E}\overline{a_{l^{\prime }m^{\prime };E}}=C_{lE}\delta
_{l}^{l^{\prime }}\delta _{m}^{m^{\prime }}\text{ , }Ea_{lm;B}\overline{%
a_{l^{\prime }m^{\prime };B}}=C_{lB}\delta _{l}^{l^{\prime }}\delta
_{m}^{m^{\prime }}\text{, }
\end{equation*}%
where
\begin{equation*}
\sum_{l}\frac{2l+1}{4\pi }C_{lE}\text{ },\text{ }\sum_{l}\frac{2l+1}{4\pi }%
C_{lB}<\infty \text{ }.
\end{equation*}

The spin (or total) angular power spectrum is defined as%
\begin{equation*}
E%
%TCIMACRO{\U{a6}}%
%BeginExpansion
{\vert}%
%EndExpansion
a_{lm;2}%
%TCIMACRO{\U{a6}}%
%BeginExpansion
{\vert}%
%EndExpansion
^{2}=:C_{l}=\left\{ C_{lE}+C_{lB}\right\} \text{ .}
\end{equation*}

In this paper, we shall be dealing with quadratic transforms of random
needlet coefficients; as in the earlier works in this area, will use the
diagram formulae (see for instance \cite{pt,s}) extensively, and we provide
here a brief overview to fix notation. Denote by $H_{q}$ the $q-$th order
Hermite polynomials, defined as
\begin{equation*}
H_{q}(u)=(-1)^{q}e^{u^{2}/2}\frac{d^{q}}{du}e^{-u^{2}/2}.
\end{equation*}%
Diagrams are basically mnemonic devices for computing the moments and
cumulants of polynomial forms in Gaussian random variables. Our notation is
the same as for instance in \cite{m,mptrf}, where again these techniques are
applied in a CMB related framework. Let $p$ and $l_{ij}=1,...,p$ be given
integers. A diagram $\gamma $ of order $(l_{1},...,l_{p})$ is a set of
points $\{(j,l):1\leq j\leq p,1\leq l\leq l_{j}\}$ called vertices, viewed
as a table $W=\overrightarrow{l_{1}}\otimes \cdots \otimes \overrightarrow{%
l_{p}}$ and a partition of these points into pairs
\begin{equation*}
\{((j,l),(k,s)):1\leq j\leq k\leq p;1\leq l\leq l_{j},1\leq s\leq l_{k}\},
\end{equation*}%
called edges. We denote by $I(W)$ the set of diagrams of order $%
(l_{1},...,l_{p})$. If the order is $l_{1}=\cdots =l_{p}=q$, for simplicity,
we also write $I(p,q)$ instead of $I(W)$. We say that:

$a)$ A diagram has a flat edge if there is at least one pair $%
\{(i,j)(i^{\prime },j^{\prime })\}$ such that $i=i^{\prime }$; we write $%
I_{F}$ for the set of diagrams that have at least one flat edge, and $I_{%
\overline{F}}$ otherwise.

$b)$ A diagram is \emph{connected} if it is not possible to partition the
rows $\overrightarrow{l_{1}}\cdots \overrightarrow{l_{p}}$ of the table $W$
into two parts, i.e. one cannot find a partition $K_{1}\cup
K_{2}=\{1,...,p\} $ that, for each member $V_{k}$ of the set of edges $%
(V_{1},...,V_{r})$ in a diagram $\gamma $, either $V_{k}\in \cup _{j\in
K_{1}}\overrightarrow{l_{j}}$, or $V_{k}\in \cup _{j\in K_{2}}%
\overrightarrow{l_{j}}$ holds; we write $I_{C}$ for connected diagrams, and $%
I_{\overline{C}}$ otherwise.

$c)$ A diagram is paired if, considering any two sets of edges $%
\{(i_{1},j_{1})(i_{2},j_{2})\}$ $\{(i_{3},j_{3})(i_{4},j_{4})\}$, then $%
i_{1}=i_{3}$ implies $i_{2}=i_{4}$; in words, the rows are completely
coupled two by two.

The following, well-known Diagram Formula will play a key role in some of
the computations to follow (see \cite{pt} and \cite{s}).

\begin{proposition}
{\label{prop2}} \emph{(Diagram Formula)} Let $(X_{1},...,X_{p})$ be a
centered Gaussian vector, and let $\gamma _{ij}=E[X_{i}X_{j}],i,j=1,...,p$
be their covariances, $H_{l_{1}},...,H_{l_{p}}$ be Hermite polynomials of
degree $l_{1},...,l_{p}$ respectively. Let $L$ be a table consisting of $p$
rows $l_{1},....l_{p}$, where $l_{j}$ is the order of Hermite polynomial in
the variable $X_{j}$. Then

\begin{equation*}
E[\prod\limits_{j=1}^{p}H_{l_{j}}(X_{j})] =\sum_{G\in
I(l_{1},...,l_{p})}\prod\limits_{1\leq i\leq j\leq p}\gamma
_{ij}^{\eta
_{ij}(G)}  \label{diag1}
\end{equation*}
\begin{equation*}
Cum(H_{l_{1}}(X_{1}),...,H_{l_{p}}(X_{p})) =\sum_{G\in
I_{c}(l_{1},...,l_{p})}\prod\limits_{1\leq i\leq j\leq p}\gamma
_{ij}^{\eta _{ij}(G)}  \notag
\end{equation*}

where, for each diagram $G$, $\eta _{ij}(G)$ is the number of edges between
rows $l_{i},l_{j}$ and $Cum(H_{l_{1}}(X_{1}),...,H_{l_{p}}(X_{p}))$ denotes
the $p$-$th$ order cumulant.
\end{proposition}

\section{Spin Needlets Spectral Estimator}

In this section, we shall establish an asymptotic result for the spectral
estimator of spin needlets in the high resolution sense, i.e. we will
investigate the asymptotic behaviour of our statistics as the frequency band
goes higher and higher. We note first, however, one very important issue. As
we mentioned earlier, spin needlet coefficients are not in general scalar
quantities. It is possible to choose a single chart to cover all points
other than the North and South Pole; these two points can be clearly
neglected without any effect on asymptotic results. The resulting spin
coefficients will in general depend on the chart, and should hence be
written as $\left\{ \beta _{R;jks}\right\} ;$ however the choice of the
chart will only produce an arbitrary phase factor $\exp (is\gamma _{k}).$
The point is that, because in this paper we are only concerned with
quadratic statistics, the phase factor is automatically lost and our
statistics for the spin spectral estimator will be invariant with respect to
the choice of coordinates. In view of this, from now on we can neglect the
issues relative to the choice of charts; we will deal with needlet
coefficients as scalar-valued complex quantities, i.e. we will take the
chart as fixed, and for notational simplicity we write $\left\{ \beta
_{jks}\right\} $ rather than $\left\{ \beta _{R;jks}\right\} .$

We begin by introducing some regularity conditions on the polarization
angular power spectrum $\Gamma _{l}$, which are basically the same as in %
\cite{gelmar}, see also \cite{bkmpAoS}, \cite{bkmpBer} and \cite%
{ejslan,fg08,spalan,mayeli} for closely related assumptions.

\begin{condition}
\label{specon}The random field $\left\{ Q+iU\right\} \left( p\right) $ is
Gaussian and isotropic with angular power spectrum such that%
\begin{equation*}
C_{l}=l^{-\alpha }g(l)>0\text{ , where }c_{0}^{-1}\leq g(l)\leq c_{0}\text{
, }\alpha >2\text{ , for all }l\in \mathbb{N}\text{ ,}
\end{equation*}%
and for every $r\in \mathbb{N}$ there exist $c_{r}>0$ such that%
\begin{equation*}
|\frac{d^{r}}{du^{r}}g(u)|\leq c_{r}u^{-r}\text{ , }u\in (\left| s\right|
,\infty )\text{ .}
\end{equation*}
\end{condition}

\begin{remark}
The condition is fulfilled for instance by angular power spectra of the form
\begin{equation*}
C_{l}=\frac{F_{1}(l)}{l^{\beta }F_{2}(l)}\text{ ,}
\end{equation*}%
where $F_{1}(l),F_{2}(l)>0$ are polynomials of degree $q_{1},q_{2}>0,$ $%
\beta +q_{2}-q_{1}=\alpha $.
\end{remark}

By (\ref{coef}), it is readily seen that%
\begin{equation*}
E\beta _{jk;s}\beta _{j^{\prime }k^{\prime };s}=
\end{equation*}%
\begin{equation*}
\sqrt{\lambda _{jk}}\sqrt{\lambda _{j^{\prime }k^{\prime }}}%
\sum_{l,l^{\prime }}b\left( \frac{\sqrt{e_{ls}}}{B^{j}}\right) b\left( \frac{%
\sqrt{e_{l^{\prime }s}}}{B^{j^{\prime }}}\right) \sum_{m,m^{\prime
}}Ea_{l;ms}a_{l^{\prime };m^{\prime }s}Y_{l;ms}\left( \xi _{jk}\right)
Y_{l^{\prime };m^{\prime }s}\left( \xi _{j^{\prime }k^{\prime }}\right) =0
\end{equation*}%
because%
\begin{equation*}
Ea_{l;ms}a_{l^{\prime };m^{\prime
}s}=Ea_{l_{1}m_{1};E}a_{l_{2}m_{2};E}+2Ea_{l_{1}m_{1};B}a_{l_{2}m_{2};E}+Ea_{l_{1}m_{1};E}a_{l_{2}m_{2};B}=0%
\text{ .}
\end{equation*}%
On the other hand, the covariance $Cov\left( \beta _{jk;s},\overline{\beta
_{jk^{\prime };s}}\right) =E\beta _{jk;s}\overline{\beta _{jk^{\prime };s}}$
is in general non-zero. In view of (\ref{coef}, it is immediate to see that%
\begin{equation}
\left| Cov\left( \beta _{jk;s},\overline{\beta _{jk^{\prime };s}}\right)
\right| =\left| \sqrt{\lambda _{jk}}\sqrt{\lambda _{jk^{\prime }}}%
\sum_{l}b^{2}\left( \frac{\sqrt{e_{ls}}}{B^{j}}\right) C_{l}\frac{\left(
2l+1\right) }{4\pi }K^{ls}\left( \xi _{jk},\xi _{jk^{\prime }}\right)
\right| \text{ ;}  \label{cov}
\end{equation}%
where
\begin{equation}
K^{ls}\left( p,p^{\prime }\right) =\sum_{m=-l}^{l}Y_{lms}\left( p\right)
\overline{Y_{lms}\left( p^{\prime }\right) }\text{ .}
\end{equation}%
For $k=k^{\prime }$ we obtain as a special case from (\ref{addup}) that
\begin{equation}
E\left| \beta _{jk;s}\right| ^{2}=\lambda _{jk}\sum_{l}b^{2}\left( \frac{%
\sqrt{e_{ls}}}{B^{j}}\right) C_{l}\frac{\left( 2l+1\right) }{4\pi }\text{ }.
\label{var}
\end{equation}%
From (\ref{cov}) and (\ref{var}) we obtain%
\begin{equation}
\left| Corr\left( \beta _{jk;s},\overline{\beta _{jk^{\prime };s}}\right)
\right| =\frac{\left| \sum_{l}b^{2}\left( \frac{\sqrt{e_{ls}}}{B^{j}}\right)
C_{l}\frac{\left( 2l+1\right) }{4\pi }K^{ls}\left( \xi _{jk},\xi
_{jk^{\prime }}\right) \right| }{\sum_{l}b\left( \frac{\sqrt{e_{ls}}}{B^{j}}%
\right) C_{l}\frac{\left( 2l+1\right) }{4\pi }}\text{ }.  \label{9}
\end{equation}%
The key result for the development of the high-frequency asymptotic theory
in the next sections is the following uncorrelation result, which was
provided by \cite{gelmar}; under Condition \ref{specon},%
\begin{equation}
\left| Corr\left( \beta _{jk;s},\overline{\beta _{jk^{\prime };s}}\right)
\right| \leq \frac{C_{M}}{\left\{ 1+B^{j}d(\xi _{jk},\xi _{jk^{\prime
}})\right\} ^{M}}\text{ , for all }M\in \mathbb{N}\text{ , some }C_{M}>0%
\text{ .}  \label{1}
\end{equation}%
The analogous result for the scalar case is due to \cite{bkmpAoS}, see also %
\cite{spalan,mayeli} for some generalizations. We recall also the following
inequality (\cite{npw2}, \emph{Lemma 4.8}), valid for some $c_{M}$ depending
only on $M$, which will be used in the following discussion:%
\begin{equation}
\sum_{k^{\prime }}\frac{1}{\left\{ 1+B^{j}d(\xi _{jk},\xi _{jk^{\prime
}})\right\} ^{M}}\frac{1}{\left\{ 1+B^{j}d(\xi _{jk^{\prime }},\xi
_{jk^{\prime \prime }})\right\} ^{M}}\leq \frac{c_{M}}{\left\{ 1+B^{j}d(\xi
_{jk},\xi _{jk^{\prime \prime }})\right\} ^{M}}\text{ }.  \label{2}
\end{equation}

In view of (\ref{var}), let us now denote%
\begin{equation*}
\Gamma _{j;s} :=\sum_{k}\left| \beta _{jk;s}\right| ^{2}=\sum_{k}\lambda
_{jk}\sum_{l}b^{2}\left( \frac{\sqrt{e_{ls}}}{B^{j}}\right) C_{l}\frac{%
\left( 2l+1\right) }{4\pi }
\end{equation*}
\begin{equation*}
=\sum_{l}b^{2}\left( \frac{\sqrt{e_{ls}}}{B^{j}}\right) C_{l}\left(
2l+1\right) \text{ .}
\end{equation*}%
Under Condition \ref{specon}, it is immediate to see that%
\begin{equation}
C_{0}B^{\left( 2-\alpha \right) j}\leq \Gamma _{j;s}\leq C_{1}B^{\left(
2-\alpha \right) j}.  \label{gamb}
\end{equation}%
A question of great practical relevance is the asymptotic behaviour of $%
\sum_{k}\left| \beta _{jk;s}\right| ^{2}$ as an estimator for $\Gamma
_{j;s}; $ for the scalar case, this issue was dealt with by \cite{bkmpAoS},
where a Functional Central Limit Theorem result is established and proposed
as a test for goodness of fit on the angular power spectrum. In \cite{pbm06}%
, the needlets estimator was applied to the cross-spectrum of CMB and Large
Scale Structure data, while \cite{fg08,fay08} have considered the presence
of missing observations and observational noise, establishing a consistency
result and providing further applications to CMB data. In the spin case,
angular power spectrum estimation was considered by \cite{gelmar}, under the
unrealistic assumptions that the spin random field $P=Q+iU$ is observed on
the whole sphere and without noise. Here we shall be concerned with the much
more realistic case where some parts of the domain $\mathbb{S}^{2}$ are
``masked''\ by the presence of foreground contamination; more precisely, we
assume data are collected only on a subset $\mathbb{S}^{2}\setminus G,$ $G$
denoting the masked region. In this section, we do not consider the presence
of observational noise, which shall be dealt with in the following section.
In the sequel, for some (arbitrary small) constant $\varepsilon >0,$ we
define $G^{\varepsilon }=\left\{ x\in \mathbb{S}^{2}:d\left( x,G\right) \leq
\varepsilon \right\} .$ Consider%
\begin{equation}
\widehat{\Gamma }_{j;sG}^{\ast }:=\left\{ \sum_{k:\xi _{jk}\in \mathbb{S}%
^{2}\backslash G^{\varepsilon }}\lambda _{k}\right\} ^{-1}\sum_{k:\xi
_{jk}\in \mathbb{S}^{2}\backslash G^{\varepsilon }}\left| \beta
_{jk;s}^{\ast }\right| ^{2}  \label{specest}
\end{equation}%
where%
\begin{equation*}
\beta _{jk;s}^{\ast }=\int_{\mathbb{S}^{2}\backslash G}P(x)\overline{\psi
_{jk;s}}(x)dx\text{ .}
\end{equation*}%
Our aim will be to prove the following

\begin{theorem}
Under condition (\ref{specon}), we have%
\begin{equation*}
\frac{\widehat{\Gamma }_{j;sG}^{\ast }-\Gamma _{j;s}}{\sqrt{Var\left\{
\widehat{\Gamma }_{j;sG}^{\ast }\right\} }}\rightarrow _{d}N(0,1)\text{ , as
}j\rightarrow \infty \text{ .}
\end{equation*}
\end{theorem}

\begin{proof}
The proof will be basically in two steps; define
\begin{equation}
\widehat{\Gamma }_{j;sG}:=\left\{ \sum_{k:\xi _{jk}\in \mathbb{S}%
^{2}\backslash G^{\varepsilon }}\lambda _{k}\right\} ^{-1}\sum_{k:\xi
_{jk}\in \mathbb{S}^{2}\backslash G^{\varepsilon }}\left| \beta
_{jk;s}\right| ^{2}\text{ ,}  \label{specestuf}
\end{equation}%
which is clearly an unfeasible version of (\ref{specest}), where the $\beta
_{jk;s}^{\ast }$ have been replaced by the coefficients (in the observed
region) evaluated without gaps. The idea will be to show that%

\begin{equation*}
\frac{\widehat{\Gamma }_{j;sG}-\Gamma _{j;s}}{\sqrt{Var\left\{ \widehat{%
\Gamma }_{j;sG}\right\} }} \rightarrow _{d}N(0,1)\text{ , }\frac{\sqrt{%
Var\left\{ \widehat{\Gamma }_{j;sG}\right\} }}{\sqrt{Var\left\{ \widehat{%
\Gamma }_{j;sG}^{\ast }\right\} }}\rightarrow 1
\end{equation*}
\begin{equation*}
\text{ and }\frac{\widehat{\Gamma }_{j;sG}^{\ast }-\widehat{\Gamma }_{j;sG}}{%
\sqrt{Var\left\{ \widehat{\Gamma }_{j;sG}^{\ast }\right\} }} \rightarrow
_{p}0\text{ , as }j\rightarrow \infty \text{ .}
\end{equation*}

The proof of these three statements is provided in separate Propositions
below.
\end{proof}

\begin{proposition}
\label{prop21} As $j\rightarrow \infty ,$ under Condition \ref{specon} we
have%
\begin{equation*}
\frac{\widehat{\Gamma }_{j;sG}-\Gamma _{j;s}}{\sqrt{Var\left\{ \widehat{%
\Gamma }_{j;sG}\right\} }}\rightarrow _{d}N(0,1)\text{ .}
\end{equation*}
\end{proposition}

\begin{proof}
Notice that
\begin{equation*}
\left( \sum_{k:\xi _{jk}\in \mathbb{S}^{2}\backslash G^{\varepsilon
}}\lambda _{k}\right) ^{2}Var\left( \widehat{\Gamma }_{j;sG}\right) =Var%
\left[ \sum_{k}\left| \beta _{jk;s}\right| ^{2}\right] =\sum_{k,k^{\prime
}}\left| E\beta _{jk;s}\overline{\beta _{jk^{\prime };s}}\right| ^{2}
\end{equation*}%
\begin{equation*}
=\sum_{k,k^{\prime }}\lambda _{jk}\lambda _{jk^{\prime }}\left|
\sum_{l}b^{2}\left( \frac{\sqrt{e_{ls}}}{B^{j}}\right) C_{l}\frac{\left(
2l+1\right) }{4\pi }K^{ls}\left( \xi _{jk},\xi _{jk^{\prime }}\right)
\right| ^{2}.
\end{equation*}%
By standard manipulations we obtain the upper bound
\begin{equation*}
Var\left[ \sum_{k}\left| \beta _{jk;s}\right| ^{2}\right]  \\
\leq C_{M}B^{2\left( 2-\alpha \right) j}\sum_{k,k^{\prime }}\lambda
_{jk}\lambda _{jk^{\prime }}\frac{1}{\left[ 1+B^{j}d(\xi _{jk},\xi
_{jk^{\prime }})\right] ^{2M}}
\end{equation*}%
\begin{equation*}
\leq C_{M}B^{2\left( 2-\alpha \right) j}\left[ \sup_{k^{\prime }}\lambda
_{jk^{\prime }}\right] \sum_{k}\lambda _{jk}\sum_{k^{\prime }}\frac{1}{\left[
1+d(\xi _{jk},\xi _{jk^{\prime }})\right] ^{2M}}
\end{equation*}%
\begin{equation*}
=\sum_{k}\lambda _{jk}O(B^{2\left( 1-\alpha \right) j})\text{ },
\end{equation*}%
in view of (\ref{1}) (\ref{gamb}) and $\lambda _{jk}\approx B^{-2j}.$ On the
other hand, we also have the trivial lower bound%
\begin{equation*}
\sum_{k,k^{\prime }}\left| E\beta _{jk;s}\overline{\beta _{jk^{\prime };s}}%
\right| ^{2}\geq \sum_{k}\left| E\beta _{jk;s}\overline{\beta _{jk;s}}%
\right| ^{2}=\Gamma _{j;s}^{2}\sum_{k}\lambda _{jk}^{2}\geq c\sum_{k}\lambda
_{jk}B^{2\left( 1-\alpha \right) j},
\end{equation*}%
whence we have%
\begin{equation}
Var\left\{ \sum_{k}\left| \beta _{jk;s}\right| ^{2}\right\} \approx \left(
\sum_{k}\lambda _{jk}\right) \left( B^{2\left( 1-\alpha \right) j}\right)
\text{ .}  \label{xiao4}
\end{equation}%
By recent results in \cite{np,petudor,noupe} it suffices to focus on
fourth-order cumulant; the proof that
\begin{equation*}
Cum_{4}\left\{ \frac{\sum_{k}\left| \beta _{jk;s}\right| ^{2}-\left(
\sum_{k}\lambda _{jk}\right) \Gamma _{j;s}}{\sqrt{Var\left\{ \sum_{k}\left|
\beta _{jk;s}\right| ^{2}\right\} }}\right\} \rightarrow 0\text{ as }%
j\rightarrow \infty \text{ ,}
\end{equation*}%
is a standard application of the Diagram Formula, indeed we have%

\begin{equation*}
Cum_{4}\left\{ \sum_{k}\left| \beta _{jk;s}\right| ^{2}-\left(
\sum_{k}\lambda _{jk}\right) \Gamma _{j;s}\right\}
\end{equation*}%
\begin{equation*}
=6\sum_{k_{1},k_{2},k_{3},k_{4}}E\beta _{jk_{1};s}\overline{\beta
_{jk_{2};s}}E\beta _{jk_{2};s}\overline{\beta _{jk_{3};s}}E\beta _{jk_{3};s}%
\overline{\beta _{jk_{4};s}}E\beta _{jk_{4};s}\overline{\beta _{jk_{1};s}}
\end{equation*}%
\begin{equation*}
\leq C\left( \Gamma _{j;s}\right) ^{4}\left( \sum_{k}\lambda _{jk}\right)
\left[ \sup_{k^{\prime }}\lambda _{jk^{\prime }}\right] ^{3}=\left(
\sum_{k}\lambda _{jk}\right) O\left( B^{\left( 2-4\alpha \right) j}\right)
\text{ ,}
\end{equation*}%

in view of (\ref{1}) and (\ref{2}). Thus the Proposition is established.
\end{proof}

Next we turn to the following

\begin{proposition}
\label{prop22} As $j\rightarrow \infty ,$ under Condition \ref{specon} we
have%
\begin{equation*}
\frac{\sqrt{Var\left\{ \widehat{\Gamma }_{j;sG}\right\} }}{\sqrt{Var\left\{
\widehat{\Gamma }_{j;sG}^{\ast }\right\} }}\rightarrow 1
\end{equation*}
\end{proposition}

\begin{proof}
Again in view of the Diagram Formula, it is enough to focus on

\begin{equation*}
Var\left( \sum_{k:\xi _{jk}\in \mathbb{S}^{2}\backslash G^{\varepsilon
}}\left| \beta _{jk;s}\right| ^{2}\right) -Var\left( \sum_{k:\xi _{jk}\in
\mathbb{S}^{2}\backslash G^{\varepsilon }}\left| \beta _{jk;s}^{\ast
}\right| ^{2}\right)
\end{equation*}
\begin{equation*}
=O(\sum_{k,k^{\prime }}\left| E\beta _{jk;s}\overline{\beta _{jk^{\prime
};s}}\right| ^{2}-\sum_{k,k^{\prime }}\left| E\beta _{jk;s}^{\ast }\overline{%
\beta _{jk^{\prime };s}^{\ast }}\right| ^{2})\text{ }.
\end{equation*}

Now notice that%
\begin{equation*}
\left| E\beta _{jk;s}\overline{\beta _{jk^{\prime };s}}\right| ^{2}-\left|
E\beta _{jk;s}^{\ast }\overline{\beta _{jk^{\prime };s}^{\ast }}\right| ^{2}
\end{equation*}%
\begin{equation*}
=E\beta _{jk;s}\overline{\beta _{jk^{\prime };s}}\left( E\overline{\beta
_{jk;s}}\beta _{jk^{\prime };s}-E\overline{\beta _{jk;s}^{\ast }}\beta
_{jk^{\prime };s}^{\ast }\right)
\end{equation*}%
\begin{equation}
+E\overline{\beta _{jk;s}^{\ast }}\beta _{jk^{\prime };s}^{\ast }\left(
E\beta _{jk;s}\overline{\beta _{jk^{\prime };s}}-E\beta _{jk;s}^{\ast }%
\overline{\beta _{jk^{\prime };s}^{\ast }}\right) ,  \label{xiao2}
\end{equation}%
and
\begin{equation*}
E\overline{\beta _{jk;s}}\beta _{jk^{\prime };s}-E\overline{\beta
_{jk;s}^{\ast }}\beta _{jk^{\prime };s}^{\ast }=E\overline{\beta _{jk;s}}%
\left( \beta _{jk^{\prime };s}-\beta _{jk^{\prime };s}^{\ast }\right)
+E\beta _{jk^{\prime };s}^{\ast }\left( \overline{\beta _{jk;s}}-\overline{%
\beta _{jk;s}^{\ast }}\right)
\end{equation*}%
\begin{equation}
\leq \left\{ E|\overline{\beta _{jk;s}}|^{2}\right\} ^{1/2}\left\{ E|\beta
_{jk^{\prime };s}-\beta _{jk^{\prime };s}^{\ast }|^{2}\right\}
^{1/2}+\left\{ E|\beta _{jk^{\prime };s}^{\ast }|^{2}\right\} ^{1/2}\left\{
E|\overline{\beta _{jk;s}}-\overline{\beta _{jk;s}^{\ast }}|^{2}\right\}
^{1/2}.  \label{xiao1}
\end{equation}%
Hence

\begin{equation*}
E|\beta _{jk;s}-\beta _{jk;s}^{\ast }|^{2} \\
\leq E\left\{ \int_{G}P(x)\overline{\psi _{jk;s}}(x)dx\right\} ^{2}\leq
E\left\{ \sup_{x\in G}\left\{ \overline{\psi _{jk;s}}(x)\right\}
\int_{G^{\varepsilon }}|P(x)|dx\right\} ^{2}
\end{equation*}%
\begin{equation*}
\leq \left[ \sup_{x\in G}\left\{ \overline{\psi _{jk;s}}(x)\right\} \right]
^{2}E\left\{ \int_{G}|P(x)|dx\right\} ^{2}
\end{equation*}%
\begin{equation*}
\leq \left[ \sup_{x\in G}\left\{ \overline{\psi _{jk;s}}(x)\right\} \right]
^{2}E\left\{ \left[ \int_{G}1dx\right] \left[ \int_{G}|P(x)|^{2}dx\right]
\right\}
\end{equation*}%
\begin{equation*}
\leq 4\pi \left[ \sup_{x\in G}\left\{ \overline{\psi _{jk;s}}(x)\right\} %
\right] ^{2}E\left\{ \left[ \int_{G}|P(x)|^{2}dx\right] \right\} =O\left(
\frac{B^{2j}}{[1+B^{j}\varepsilon ]^{2M}}\right) \text{ .}
\end{equation*}%

Now recall that%
\begin{equation*}
E|\beta _{jk;s}|^{2}=O\left( B^{-\alpha j}\right) \text{ },
\end{equation*}%
whence $E|\beta _{jk;s}^{\ast }|^{2}=O\left( B^{-\alpha j}\right) ,$ if $%
M>\alpha /2+1.$ Hence, in view of (\ref{xiao1})%
\begin{equation}
\left| E\overline{\beta _{jk;s}^{\ast }}\beta _{jk^{\prime };s}^{\ast }-E%
\overline{\beta _{jk;s}}\beta _{jk^{\prime };s}\right| \leq \frac{CB^{\left(
1-\alpha /2\right) j}}{[1+B^{j}\varepsilon ]^{M}}\text{ },  \label{xiao3}
\end{equation}%
for some constant $C>0.$ Also, from (\ref{xiao2}) and (\ref{xiao3}) we
obtain that%
\begin{equation*}
\sum_{k,k^{\prime }}\left( \left| E\beta _{jk;su}\overline{\beta
_{jk^{\prime };su}}\right| ^{2}-\left| E\beta _{jk;su}^{\ast }\overline{%
\beta _{jk^{\prime };su}^{\ast }}\right| ^{2}\right)
\end{equation*}%
\begin{equation*}
\leq \sum_{k,k^{\prime }}\left( \left| E\beta _{jk;s}\overline{\beta
_{jk^{\prime };s}}\right| +\left| E\beta _{jk;s}^{\ast }\overline{\beta
_{jk^{\prime };s}^{\ast }}\right| \right) O\left( \frac{B^{-j\alpha /2}}{%
[1+B^{j}\varepsilon ]^{M}}\right)
\end{equation*}%
\begin{equation*}
\leq O\left( \frac{B^{\left( 1-\alpha /2\right) j}}{[1+B^{j}\varepsilon
]^{M}}\right) \Gamma _{j;s}\sum_{k,k^{\prime }}\frac{C_{M}\sqrt{\lambda
_{jk}\lambda _{jk^{\prime }}}}{\left\{ 1+B^{j}d(\xi _{jk},\xi _{jk^{\prime
}})\right\} ^{M}}
\end{equation*}%
\begin{equation*}
\leq O\left( \frac{B^{3\left( 1-\alpha /2\right) j}}{[1+B^{j}\varepsilon
]^{M}}\right) \sum_{k}\lambda _{jk}\text{ }.
\end{equation*}%

Recall from (\ref{xiao4}) that $Var\left( \sum_{k:\xi _{jk}\in \mathbb{S}%
^{2}\backslash G}\left| \beta _{jk;s}\right| ^{2}\right) =\left(
\sum_{k}\lambda _{jk}\right) O\left( B^{2\left( 1-\alpha \right) j}\right) .$
Hence for $M$ large enough, that is $M>1+\alpha /2,$ the statement of the
Proposition is established.
\end{proof}

\begin{proposition}
\label{prop23} As $j\rightarrow \infty ,$ under Condition \ref{specon} we
have%
\begin{equation*}
\frac{\widehat{\Gamma }_{j;sG}^{\ast }-\widehat{\Gamma }_{j;sG}}{\sqrt{%
Var\left\{ \widehat{\Gamma }_{j;sG}^{\ast }\right\} }}\rightarrow _{p}0\text{
.}
\end{equation*}
\end{proposition}

\begin{proof}
We have%
\begin{equation*}
E\left\{ \left[ \sum_{k:\xi _{jk}\in \mathbb{S}^{2}\backslash G^{\varepsilon
}}\lambda _{k}\right] \left( \widehat{\Gamma }_{j;sG}^{\ast }-\widehat{%
\Gamma }_{j;sG}\right) \right\} ^{2}=E\left\{ \sum_{k}\left| \beta
_{jk;s}\right| ^{2}-\left| \beta _{jk;s}^{\ast }\right| ^{2}\right\} ^{2},
\end{equation*}%
which we can expand as follows%

\begin{equation*}
E\left\{ \sum_{k}\overline{\beta _{jk;s}}\left( \beta _{jk;s}-\beta
_{jk;s}^{\ast }\right) +\sum_{k}\beta _{jk;s}^{\ast }\left( \overline{\beta
_{jk;s}}-\overline{\beta _{jk;s}^{\ast }}\right) \right\} ^{2}
\end{equation*}%
\begin{equation*}
=E\left\{ \sum_{k}\overline{\beta _{jk;s}}\left( \beta _{jk;s}-\beta
_{jk;s}^{\ast }\right) \right\} ^{2}+E\left\{ \sum_{k}\beta _{jk;s}^{\ast
}\left( \overline{\beta _{jk;s}}-\overline{\beta _{jk;s}^{\ast }}\right)
\right\} ^{2}
\end{equation*}%
\begin{equation*}
+2E\left\{ \sum_{k}\overline{\beta _{jk;s}}\left( \beta _{jk;s}-\beta
_{jk;s}^{\ast }\right) \right\} \left\{ \sum_{k}\beta _{jk;s}^{\ast }\left(
\overline{\beta _{jk;s}}-\overline{\beta _{jk;s}^{\ast }}\right) \right\}
\end{equation*}%

\begin{equation*}
=\sum_{k,k^{\prime }}\left[ E\overline{\beta _{jk;s}}\left( \beta
_{jk^{\prime };s}-\beta _{jk^{\prime };s}^{\ast }\right) E\overline{\beta
_{jk^{\prime };s}}\left( \beta _{jk;s}-\beta _{jk;s}^{\ast }\right) \right.
\end{equation*}%
\begin{equation*}
\left. +E\beta _{jk;s}^{\ast }\left( \overline{\beta _{jk^{\prime };s}}-%
\overline{\beta _{jk^{\prime };s}^{\ast }}\right) E\beta _{jk^{\prime
};s}^{\ast }\left( \overline{\beta _{jk;s}}-\overline{\beta _{jk;s}^{\ast }}%
\right) \right]
\end{equation*}%
\begin{equation*}
+\left\{ \sum_{k}E\overline{\beta _{jk;s}}\left( \beta _{jk;s}-\beta
_{jk;s}^{\ast }\right) \right\} ^{2}+\left\{ \sum_{k}E\beta _{jk;s}^{\ast
}\left( \overline{\beta _{jk;s}}-\overline{\beta _{jk;s}^{\ast }}\right)
\right\} ^{2}
\end{equation*}%
\begin{equation*}
+2\left\{ \sum_{k}E\overline{\beta _{jk;s}}\left( \beta _{jk;s}-\beta
_{jk;s}^{\ast }\right) \right\} \left\{ \sum_{k}E\beta _{jk;s}^{\ast }\left(
\overline{\beta _{jk;s}}-\overline{\beta _{jk;s}^{\ast }}\right) \right\}
\end{equation*}%
\begin{equation*}
+2\left\{ \sum_{k,k^{\prime }}E\overline{\beta _{jk;s}}\left( \beta
_{jk^{\prime };s}-\beta _{jk^{\prime };s}^{\ast }\right) E\beta _{jk^{\prime
};s}^{\ast }\left( \overline{\beta _{jk;s}}-\overline{\beta _{jk;s}^{\ast }}%
\right) \right\}
\end{equation*}%
\begin{equation*}
+2\left\{ \sum_{k,k^{\prime }}E\overline{\beta _{jk;s}}\beta _{jk^{\prime
};s}^{\ast }E\left( \overline{\beta _{jk;s}}-\overline{\beta _{jk;s}^{\ast }}%
\right) \left( \beta _{jk^{\prime };s}-\beta _{jk^{\prime };s}^{\ast
}\right) \right\} \text{ .}
\end{equation*}%

Now recall again
\begin{equation*}
E\left| \beta _{jk;s}\right| ^{2},E\left| \beta _{jk;s}^{\ast }\right|
^{2}\leq CB^{-\alpha j},\text{ and }E\left| \beta _{jk;s}-\beta
_{jk;s}^{\ast }\right| ^{2}\leq \frac{C^{\prime }B^{2j}}{[1+B^{j}\varepsilon
]^{M}}\text{ },
\end{equation*}%
whence from the same steps as in the previous Proposition, we have
\begin{equation*}
E\overline{\beta _{jk;s}}\left( \beta _{jk^{\prime };s}-\beta _{jk^{\prime
};s}^{\ast }\right) ,E\overline{\beta _{jk;s}^{\ast }}\left( \beta
_{jk^{\prime };s}-\beta _{jk^{\prime };s}^{\ast }\right) \leq \frac{%
CB^{\left( 1-\alpha /2\right) j}}{[1+B^{j}\varepsilon ]^{M}}\text{ }.
\end{equation*}%
It follows that%
\begin{equation*}
E\left\{ \sum_{k}\left| \beta _{jk;s}\right| ^{2}-\left| \beta _{jk;s}^{\ast
}\right| ^{2}\right\} ^{2}\leq \frac{CB^{\left( 6-\alpha \right) j}}{%
[1+B^{j}\varepsilon ]^{2M}}\text{ }.
\end{equation*}%
By arguments in the previous Propositions, we know that%
\begin{equation*}
Var\left\{ \left[ \sum_{k:\xi _{jk}\in \mathbb{S}^{2}\backslash
G^{\varepsilon }}\lambda _{k}\right] \widehat{\Gamma }_{j;sG}^{\ast
}\right\} \approx \left( \sum_{k:\xi _{jk}\in \mathbb{S}^{2}\backslash
G^{\varepsilon }}\lambda _{jk}\right) B^{2\left( 1-\alpha \right) j};
\end{equation*}%
thus the statement is established, provided we take $M>2+\alpha /2.$
\end{proof}

\begin{remark}
In general the expression for $Var\left\{ \widehat{\Gamma }_{j;sG}\right\}
,Var\left\{ \widehat{\Gamma }_{j;sG}^{\ast }\right\} $ depends on the
unknown angular power spectrum. However, the normalizing factors can be
consistently estimated by subsampling techniques, following the same steps
as in \cite{bkmpBer}.
\end{remark}

\section{Detection of asymmetries}

In this Section, we shall consider one more possible application of spin
needlets to problems of interest for Cosmology. In particular, a highly
debated issue in modern Cosmology relates to the existence of ``features'',
i.e. asymmetries in the distribution of CMB radiation (for instance between
the Northern and the Southern hemispheres, in Galactic coordinates). These
issues have been the object of dozens of physical papers, in the last few
years, some of them exploiting scalar needlets, see \cite{pietrobon1}.

In order to investigate this issue, we shall employ a similar technique as %
\cite{bkmpBer} for the scalar case. More precisely, we shall focus on the
difference between the estimated angular power spectrum over two different
regions of the sky. Let us consider $A_{1},A_{2},$ two subsets of $\mathbb{S}%
^{2}$ such that $A_{1}\cap A_{2}=\emptyset ;$ we do not assume that $%
A_{1}\cup A_{2}=\mathbb{S}^{2},$ i.e. we admit the presence of missing
observations. For practical applications, $A_{1}$ and $A_{2}$ can be
visualized as the spherical caps centered at the north and south pole $N,S$
(i.e. $A_{1}=\left\{ x\in \mathbb{S}^{2}:d(x,N)\leq \pi /2\right\} ,$ $%
A_{2}=\left\{ x\in \mathbb{S}^{2}:d(x,S)\leq \pi /2\right\} ,$ but the
results would hold without any modification for general subsets and could be
easily generalized to a higher number of regions. We shall then focus on the
statistic%
\begin{equation*}
\frac{\widehat{\Gamma }_{j;sA_{1}}^{\ast }-\widehat{\Gamma }%
_{j;sA_{2}}^{\ast }}{\sqrt{Var\left\{ \widehat{\Gamma }_{j;sA_{1}}^{\ast
}\right\} +Var\left\{ \widehat{\Gamma }_{j;sA_{2}}^{\ast }\right\} }}\text{ ,%
}
\end{equation*}%
where

\begin{equation*}
\widehat{\Gamma }_{j;sA_{1}}^{\ast } :=\left\{ \sum_{k:\xi _{jk}\in
A_{1}^{\varepsilon }}\lambda _{k}\right\} ^{-1}\sum_{k:\xi _{jk}\in
A_{1}^{\varepsilon }}\left| \beta _{jk;s}^{\ast }\right| ^{2},\text{ }
\end{equation*}
\begin{equation*}
\widehat{\Gamma }_{j;sA_{2}}^{\ast } :=\left\{ \sum_{k:\xi _{jk}\in
A_{2}^{\varepsilon }}\lambda _{k}\right\} ^{-1}\sum_{k:\xi _{jk}\in
A_{2}^{\varepsilon }}\left| \beta _{jk;s}^{\ast }\right| ^{2}\text{ , some }%
\varepsilon >0\text{ .}
\end{equation*}

We are here able to establish the following

\begin{proposition}
\label{features}As $j\rightarrow \infty $ , we have%
\begin{equation*}
\left(
\begin{array}{c}
\left[ Var\left\{ \widehat{\Gamma }_{j;sA_{1}}^{\ast }\right\} \right]
^{-1/2}\left( \widehat{\Gamma }_{j;sA_{1}}^{\ast }-\Gamma _{j;s}\right) \\
\left[ Var\left\{ \widehat{\Gamma }_{j;sA_{2}}^{\ast }\right\} \right]
^{-1/2}\left( \widehat{\Gamma }_{j;sA_{2}}^{\ast }-\Gamma _{j;s}\right)%
\end{array}%
\right) \rightarrow _{d}N(0_{2},I_{2})\text{ ,}
\end{equation*}%
where $(0_{2},I_{2})$ are, respectively, the $2\times 1$ vector of zeros and
the $2\times 2$ identity matrix.
\end{proposition}

\begin{proof}
By the Cramer-Wold device, the proof can follow very much the same steps as
for the univariate case. We first establish the asymptotic uncorrelation of
the two components, i.e. we show that%
\begin{equation}
\lim_{j\rightarrow \infty }\left[ Var\left\{ \widehat{\Gamma }%
_{j;sA_{1}}^{\ast }\right\} Var\left\{ \widehat{\Gamma }_{j;sA_{2}}^{\ast
}\right\} \right] ^{-1/2}E\left\{ \left( \widehat{\Gamma }_{j;sA_{1}}^{\ast
}-\Gamma _{j;s}\right) \left( \widehat{\Gamma }_{j;sA_{2}}^{\ast }-\Gamma
_{j;s}\right) \right\} =0\text{ .}  \label{7}
\end{equation}%
Now%
\begin{equation*}
E\left( \widehat{\Gamma }_{j;sA_{1}}^{\ast }-\Gamma _{j;s}\right) \left(
\widehat{\Gamma }_{j;sA_{2}}^{\ast }-\Gamma _{j;s}\right) =E\left( \widehat{%
\Gamma }_{j;sA_{1}}^{\ast }-\Gamma _{j;s}\right) E\left( \widehat{\Gamma }%
_{j;sA_{2}}^{\ast }-\Gamma _{j;s}\right)
\end{equation*}%
\begin{equation}
+\left\{ \sum_{k:\xi _{jk}\in \mathbb{S}^{2}\backslash A_{1}^{\varepsilon
}}\lambda _{k}\sum_{k:\xi _{jk}\in \mathbb{S}^{2}\backslash
A_{2}^{\varepsilon }}\lambda _{k}\right\} ^{-1}\sum_{k:\xi _{jk}\in \mathbb{S%
}^{2}\backslash A_{1}^{\varepsilon }}\sum_{k^{\prime }:\xi _{jk^{\prime
}}\in \mathbb{S}^{2}\backslash A_{2}^{\varepsilon }}\left| E\beta
_{jk;s}^{\ast }\overline{\beta _{jk^{\prime };s}^{\ast }}\right| ^{2}.
\label{6}
\end{equation}%
In view of (\ref{1}) and Proposition \ref{prop23}, we have

\begin{equation*}
\left| (\ref{6})\right| \leq \left( \Gamma _{j;s}\right) ^{2}\sum_{k:\xi
_{jk}\in \mathbb{S}^{2}\backslash A_{1}^{\varepsilon }}\sum_{k^{\prime }:\xi
_{jk^{\prime }}\in \mathbb{S}^{2}\backslash A_{2}^{\varepsilon }}\frac{%
C\lambda _{jk}\lambda _{jk^{\prime }}}{\left[ 1+B^{j}d(\xi _{jk},\xi
_{jk^{\prime }})\right] ^{2M}}
\end{equation*}
\begin{equation*}
\leq \frac{C\left( \Gamma _{j;s}\right) ^{2}\left[ \sup_{k}\lambda _{jk}%
\right] ^{2}}{\left[ 1+2B^{j}\varepsilon \right] ^{2\left( M-1\right) }}%
=O\left( B^{2\left( 1-\alpha -M\right) j}\right) .
\end{equation*}
Thus (\ref{7})\ is established, in view of (\ref{xiao4}) and Propositions (%
\ref{prop22}), (\ref{prop23}). For the fourth order cumulant, given any
generic constants $u,v$, we shall write
\begin{equation}
X=u\left[ Var\left\{ \widehat{\Gamma }_{j;sA_{1}}^{\ast }\right\} \right]
^{-1/2}\left( \widehat{\Gamma }_{j;sA_{1}}^{\ast }-\Gamma _{j;s}\right) ,
\label{xiao10}
\end{equation}%
and%
\begin{equation}
Y=v\left[ Var\left\{ \widehat{\Gamma }_{j;sA_{2}}^{\ast }\right\} \right]
^{-1/2}\left( \widehat{\Gamma }_{j;sA_{2}}^{\ast }-\Gamma _{j;s}\right)
\text{ .}  \label{xiao11}
\end{equation}%
Recall that%

\begin{equation*}
Cum_{4}\left( X+Y\right) =Cum_{4}\left( X\right) +Cum_{4}\left( Y\right)
+4Cum(X,Y,Y,Y)
\end{equation*}
\begin{equation*}
+6Cum(X,X,Y,Y)+4Cum(X,X,X,Y)\text{ };
\end{equation*}

by results in the previous Section, we have immediately $Cum_{4}\left(
X\right) ,$ $Cum_{4}\left( Y\right) \rightarrow 0,$ as $j\rightarrow \infty
. $ On the other hand, in view of Proposition \ref{prop23} and the
equivalence between convergence in probability and in $L^{p}$ for Gaussian
subordinated processes (see \cite{JAN}), we can replace $\widehat{\Gamma }%
_{j;sA_{i}}^{\ast }$ by $\widehat{\Gamma }_{j;sA_{i}}$ in (\ref{xiao10}) and
(\ref{xiao11}), and we have easily%

\begin{equation*}
Cum(X,Y,Y,Y)
\end{equation*}%
\begin{equation*}
\leq CB^{4\left( \alpha -1\right) j}\left( \Gamma _{j;s}\right)
^{2}\sum_{k:\xi _{jk}\in \mathbb{S}^{2}\backslash A_{1}^{\varepsilon
}}\sum_{\xi _{jk_{1}},..,\xi _{jk_{3}}\in \mathbb{S}^{2}\backslash
A_{2}^{\varepsilon }}\frac{\lambda _{jk}\lambda _{jk_{1}}\lambda
_{jk_{3}}\lambda _{jk_{3}}}{\left[ 1+B^{j}d(\xi _{jk},\xi _{jk_{1}})\right]
^{M}}
\end{equation*}%
\begin{equation*}
\times \frac{1}{\left[ 1+B^{j}d(\xi _{jk_{2}},\xi _{jk_{1}})\right] ^{M}\left[
1+B^{j}d(\xi _{jk_{3}},\xi _{jk_{2}})\right] ^{M}\left[ 1+B^{j}d(\xi
_{jk},\xi _{jk_{3}})\right] ^{M}}
\end{equation*}%
\begin{equation*}
\leq \frac{CB^{4\left( \alpha -1\right) j}\left( \Gamma _{j;s}\right) ^{2}%
\left[ \sup_{k}\lambda _{jk}\right] ^{4}}{\left[ 1+2B^{j}\varepsilon \right]
^{2\left( M-1\right) }}=O\left( B^{-2\left( M+1\right) j}\right) .
\end{equation*}%

Similarly, we have
\begin{equation*}
Cum(X,X,X,Y),Cum(X,X,Y,Y)\leq CB^{-2\left( M+1\right) j}.
\end{equation*}%
Thus the Proposition is established, provided we choose $M>2+\alpha .$
\end{proof}

\begin{remark}
An obvious consequence of Proposition \ref{features} is
\begin{equation*}
\frac{\widehat{\Gamma }_{j;sA_{1}}^{\ast }-\widehat{\Gamma }%
_{j;sA_{2}}^{\ast }}{\sqrt{Var\left\{ \widehat{\Gamma }_{j;sA_{1}}^{\ast
}\right\} +Var\left\{ \widehat{\Gamma }_{j;sA_{2}}^{\ast }\right\} }}%
\rightarrow _{d}N(0,1)\text{ .}
\end{equation*}%
This result provides the asymptotic justification to implement on
polarization data the same testing procedures as those considered for
instance by \cite{pietrobon1} to search for features and asymmetries in CMB\
scalar data; i.e., it is possible to estimate for instance the angular power
spectrum on the Northern and Southern hemisphere and test whether they are
statistically different, as suggested by some empirical findings of the
recent cosmological literature.
\end{remark}

\section{\protect\bigskip Estimation with noise}

In the previous sections, we worked under a simplifying assumption, i.e. we
figured that although observations on some parts of the sphere were
completely unattainable, data on the remaining part were available free of
noise. In this Section, we aim at relaxing this assumption; in particular,
we shall consider the more realistic circumstances where, while we still
take some regions of the sky to be completely unobservable, even for those
where observations are available the latter are partially contaminated by
noise.

To understand our model for noise, we need to review a view basic facts on
the underlying physics. A key issue about (scalar and polarized) CMB\
radiation experiments is that they actually measure radiation across a set
of different electromagnetic frequencies, ranging from 30 GHz to nearly 900.
One of the key predictions of Cosmology, whose experimental confirmation led
to the Nobel Prize for J.Mather in 2006, is that CMB radiation in all its
components follows a blackbody pattern of dependence over frequency. More
precisely, the intensity $I_{A}$ is distributed along to the various
frequencies according to the Planckian curve of blackbody emission

\begin{equation}
I_{A}(v,P)=\frac{2h\nu ^{3}}{c^{2}}\frac{1}{\exp (\frac{h\nu }{k_{B}A})-1}%
\text{ .}  \label{planck}
\end{equation}%
Here, $A$ is a scalar quantity which is the only free parameter in (\ref%
{planck}), and therefore uniquely determines the shape of the curve: we have
$A=T$ for the traditional temperature data, whereas for polarization
measurements one can take $A=Q,U.$ Now the point is that, although there are
also a number of foreground sources (such as galaxies or intergalactic dust)
that emit radiation on these frequencies; all these astrophysical components
(other than CMB) do not follow a blackbody curve.

We shall hence assume that $D$ detectors are available at frequencies $\nu
_{1},...,\nu _{D}$, so that the following vector random field is observed:%
\begin{equation*}
P_{v_{r}}(x)=P(x)+N_{v_{r}}(x)\text{ ;}
\end{equation*}%
here, both $P(x),N_{v}(x)$ are taken to be Gaussian zero-mean, mean square
continuous random fields, independent among them and such that, while the
signal $P(x)$ is identical across all frequencies, the noise $N_{v}(x)$ is
not. More precisely, we shall assume for noise the same regularity
conditions as for the signal $P,$ again under the justification that they
seem mild and general:

\begin{condition}
\label{specN}The (spin) random field $N_{v}(x)$ is Gaussian and isotropic,
independent from $P(x)$ and with total angular power spectrum $\left\{
C_{lN}\right\} $ such that%
\begin{equation*}
C_{lN}=l^{-\gamma }g_{N}(l)>0\text{ , where }c_{0N}^{-1}\leq g_{N}(l)\leq
c_{0N}\text{ , }\gamma >2\text{ , }l\in \mathbb{N}\text{ ,}
\end{equation*}%
and for every $r\in \mathbb{N}$ there exist $c_{r}>0$ such that%
\begin{equation*}
|\frac{d^{r}}{du^{r}}g_{N}(u)|\leq c_{rN}u^{-r}\text{ , }u\in (\left|
s\right| ,\infty )\text{ .}
\end{equation*}
\end{condition}

It follows from our previous assumptions that for each frequency $\nu _{r}$
we shall be able to evaluate%
\begin{equation*}
\int_{\mathbb{S}^{2}}P_{v_{r}}(x)\overline{\psi _{jk;s}}(x)dx=:\beta
_{jk;sr}=\beta _{jk;sP}+\beta _{jk;sN_{r}}
\end{equation*}%
where clearly%
\begin{equation*}
\beta _{jk;sP}=\int_{\mathbb{S}^{2}}P(x)\overline{\psi _{jk;s}}(x)dx\text{ ,
}\beta _{jk;sN_{r}}=\int_{\mathbb{S}^{2}}N_{v_{r}}(x)\overline{\psi _{jk;s}}%
(x)dx\text{ .}
\end{equation*}%
Now it is immediate to note that%
\begin{equation*}
E|\beta _{jk;sr}|^{2} =E|\beta _{jk;sP}+\beta _{jk;sN_{r}}|^{2}
\end{equation*}%
\begin{equation*}
=E\beta _{jk;sP}\overline{\beta _{jk;sP}}+E\beta _{jk;sN_{r}}\overline{%
\beta _{jk;sN_{r}}}+E\beta _{jk;sN_{r}}\overline{\beta _{jk;sP}}+E\beta
_{jk;sP}\overline{\beta _{jk;sN_{r}}}
\end{equation*}%
\begin{equation*}
=E|\beta _{jk;sP}|^{2}+E|\beta _{jk;sN_{r}}|^{2}\text{ ,}
\end{equation*}%

so that the estimator $\sum_{k}|\beta _{jk;sr}|^{2}$ will now be upward
biased. In the next subsections we shall discuss two possible solutions for
dealing with this bias terms, along the lines of (\cite{polenta}), and we
will provide statistical procedures to test for estimation bias. We note
first that correlation of needlet coefficients across different channels are
provided by%
\begin{equation*}
E\beta _{jk;sr}\overline{\beta _{jk^{\prime };sr}}=E\beta _{jk;sP}\overline{%
\beta _{jk^{\prime };sP}}+E\beta _{jk;sN_{r}}\beta _{jk^{\prime };sN_{r}}%
\text{ .}
\end{equation*}%
Denote%
\begin{equation*}
\Gamma _{j;s}^{N}=\sum_{k}E|\beta _{jk;sN_{r}}|^{2}=\sum_{l}b^{2}(\frac{%
\sqrt{e_{ls}}}{B^{j}})\frac{2l+1}{4\pi }C_{lN}\text{ };
\end{equation*}%
as before, it is easy to obtain that $C_{1}B^{\left( 2-\gamma \right) j}\leq
\Gamma _{j;s}^{N}\leq C_{2}B^{\left( 2-\gamma \right) j}.$ With the same
discussion as for (\ref{1}) provided by \cite{gelmar}, we have that, under
Condition \ref{specon} and \ref{specN},
\begin{equation}
\left| Corr\left( \beta _{jk;sr},\overline{\beta _{jk^{\prime };sr}}\right)
\right| \leq \frac{C_{M}}{\left\{ 1+B^{j}d(\xi _{jk},\xi _{jk^{\prime
}})\right\} ^{M}}\text{ , for all }M\in \mathbb{N}\text{ }.  \label{4}
\end{equation}

\subsection{The Needlet Auto-Power Spectrum Estimator}

In many circumstances, it can be reasonable to assume that the angular power
spectrum of the noise component, $C_{lN},$ is known in advance to the
experimenter. For instance, if noise is primarily dominated by instrumental
components, then its behaviour may possibly be calibrated before the
experimental devices are actually sent in orbit, or otherwise by observing a
peculiar region where the signal has been very tightly measured by previous
experiments. Assuming the angular power spectrum of noise to be known, the
expected value for the bias term is immediately derived:%
\begin{equation*}
E|\beta _{jk;sN_{r}}|^{2}=\sum_{l}b^{2}(\frac{\sqrt{e_{ls}}}{B^{j}})\frac{%
2l+1}{4\pi }C_{lN_{r}}\text{ ,}
\end{equation*}%
whence it is natural to propose the bias-corrected estimator

\begin{equation*}
\widetilde{\Gamma }_{j}^{AP} :=\frac{1}{D}\sum_{k}\sum_{r}\left\{ \left|
\beta _{jk;sr}\right| ^{2}-E|\beta _{jk;sN_{r}}|^{2}\right\}
\end{equation*}%
\begin{equation*}
=\frac{1}{D}\sum_{k}\sum_{r}\left\{ \left( \beta _{jk;sP}+\beta
_{jk;sN_{r}}\right) \left( \overline{\beta _{jk;sP}}+\overline{\beta
_{jk;sN_{r}}}\right) -E|\beta _{jk;sN_{r}}|^{2}\right\}
\end{equation*}%
\begin{equation*}
=\sum_{k}\left| \beta _{jk;sP}\right| ^{2}+\frac{1}{D}\left\{
\sum_{k}\sum_{r}\left( \beta _{jk;sP}\overline{\beta _{jk;sN_{r}}}+\beta
_{jk;sN_{r}}\overline{\beta _{jk;sP}}+\left[ \left| \beta
_{jk;sN_{r}}\right| ^{2}-E|\beta _{jk;sN_{r}}|^{2}\right] \right) \right\}
\text{ .}
\end{equation*}%

We label the previous statistics the needlet auto-power spectrum estimator
(AP, compare \cite{polenta}). The derivation of the following Proposition is
rather standard, and hence omitted for brevity's sake.

\begin{proposition}
As $j\rightarrow \infty ,$ we have
\begin{equation*}
\frac{\widetilde{\Gamma }_{j}^{AP}-\Gamma _{j}}{\sqrt{Var\left\{ \widetilde{%
\Gamma }_{j}^{AP}\right\} }}\rightarrow _{d}N(0,1)\text{ ,}
\end{equation*}%
where%
\begin{equation*}
Var\left\{ \widetilde{\Gamma }_{j}^{AP}\right\} =O(B^{2\left( 1-\min \left(
\alpha ,\gamma \right) \right) j})\text{ .}
\end{equation*}
\end{proposition}

As before, the normalizing variance in the denominator can be consistently
estimated by subsampling techniques, along the lines of \cite{bkmpBer}. It
should be noticed that the rate of convergence for $\left\{ \widetilde{%
\Gamma }_{j}^{AP}-\Gamma _{j}\right\} =O(B^{\left( 1-\min \left( \alpha
,\gamma \right) \right) j})$ is the same as in the noiseless case for $%
\gamma \geq \alpha ,$ whereas it slower otherwise, when the noise is
asymptotically dominating. The ``signal-to-noise'' ratio $\Gamma _{j}/\sqrt{%
Var\left\{ \widetilde{\Gamma }_{j}^{AP}\right\} }$ is easily seen to be in
the order of $B^{2j-\alpha j}/B^{\left( 1-\min \left( \alpha ,\gamma \right)
\right) j}=B^{j(1+\min \left( \alpha ,\gamma \right) -\alpha )},$ whence it
decays to zero unless $\alpha \leq \gamma +1.$

\subsection{The Needlet Cross-Power Spectrum estimator}

To handle the bias term, we shall pursue here a different strategy than the
previous subsection, dispensing with any prior knowledge of the spectrum of
the noise component. The idea is to exploit the fact that, while the signal
is perfectly correlated among the different frequency components, noise is
by assumption independent. We shall hence focus on the needlets
cross-angular power spectrum estimator (CP), defined as%

\begin{equation*}
\widetilde{\Gamma }_{j}^{CP} :=\frac{1}{D(D-1)}\sum_{k}\sum_{r_{1}\neq
r_{2}}\beta _{jk;sr_{1}}\overline{\beta _{jk;sr_{2}}}
\end{equation*}%
\begin{equation*}
=\frac{1}{D(D-1)}\sum_{k}\sum_{r_{1}\neq r_{2}}\left( \beta _{jk;sP}+\beta
_{jk;sN_{r_{1}}}\right) \left( \overline{\beta _{jk;sP}}+\overline{\beta
_{jk;sN_{r_{2}}}}\right)
\end{equation*}%
\begin{equation*}
=\sum_{k}\left| \beta _{jk;sP}\right| ^{2}+\frac{1}{D(D-1)}\left\{
\sum_{k}\sum_{r_{1}\neq r_{2}}\left( \beta _{jk;sP}\overline{\beta
_{jk;sN_{r_{2}}}}+\beta _{jk;sN_{r_{1}}}\overline{\beta _{jk;sP}}+\beta
_{jk;sN_{r_{1}}}\overline{\beta _{jk;sN_{r_{2}}}}\right) \right\} \text{ .}
\end{equation*}%

In view of the previous independence assumptions, it is then immediately
seen that the above estimator is unbiased for $\Gamma _{j},$ i.e.
\begin{equation*}
E\widetilde{\Gamma }_{j}^{CP}=\sum_{k}E\left| \beta _{jk;sP}\right|
^{2}=\sum_{l}b^{2}(\frac{\sqrt{e_{ls}}}{B^{j}})\frac{2l+1}{4\pi }C_{l}\text{
.}
\end{equation*}%
We are actually able to establish a stronger result, namely

\begin{proposition}
As $j\rightarrow \infty ,$ we have
\begin{equation*}
\frac{\widetilde{\Gamma }_{j}^{CP}-\Gamma _{j}}{\sqrt{Var\left\{ \widetilde{%
\Gamma }_{j}^{CP}\right\} }}\rightarrow _{d}N(0,1)\text{ , }Var\left\{
\widetilde{\Gamma }_{j}^{CP}\right\} =O(B^{2\left( 1-\min (\alpha ,\gamma
)\right) j})\text{ .}
\end{equation*}
\end{proposition}

We omit also this (standard) proof for brevity's sake. We can repeat here
the same comments as in the previous subsection, concerning the possibility
of estimating the normalizing variance by subsampling techniques, along the
lines of \cite{bkmpBer}, and the roles of $\alpha $,$\gamma $ for the rate
of convergence $\left\{ \widetilde{\Gamma }_{j}^{CP}-\Gamma _{j}\right\}
=O(B^{\left( 1-\min \left( \alpha ,\gamma \right) \right) j})$.

\subsection{Hausman Test for Noise Misspecification}

In the previous two subsections, we have considered two alternate estimators
for the angular power spectrum, in the presence of observational noise. It
is a standard result (compare \cite{polenta}) that the auto-power spectrum
estimator enjoys a smaller variance, provided of course that the model for
noise is correct. Loosely speaking, we can hence conclude that the
auto-power spectrum estimator is more efficient when noise is correctly
specified, while the cross-power spectrum estimator is more robust, as it
does not depend on any previous knowledge on the noise angular power
spectrum. An obvious question at this stage is whether the previous results
can be exploited to implement a procedure to search consistently for noise
misspecification. The answer is indeed positive, as we shall show along the
lines of the procedure suggested by \cite{polenta}.

\begin{proposition}
Under Assumptions \ref{specon} and \ref{specN} , we have
\begin{equation*}
\frac{\widetilde{\Gamma }_{j;s}^{CP}-\widetilde{\Gamma }_{j;s}^{AP}}{\sqrt{%
Var\left\{ \widetilde{\Gamma }_{j;s}^{CP}-\widetilde{\Gamma }%
_{j;s}^{AP}\right\} }}\rightarrow _{d}N(0,1),
\end{equation*}%
where%
\begin{equation*}
Var\left\{ \widetilde{\Gamma }_{j;s}^{CP}-\widetilde{\Gamma }%
_{j;s}^{AP}\right\} =O(B^{2\left( 1-\gamma \right) j})
\end{equation*}
\end{proposition}

\begin{proof}
The proof is again quite standard, and we only need to provide the main
details. Notice first that

\begin{equation*}
\widetilde{\Gamma }_{j;s}^{CP}-\widetilde{\Gamma }_{j;s}^{AP} =\frac{1}{%
D(D-1)}\sum_{k}\sum_{r_{1}\neq r_{2}}\beta _{jk;sr_{1}}\overline{\beta
_{jk;sr_{2}}}-\frac{1}{D}\sum_{k}\sum_{r}\left\{ \left| \beta
_{jk;sr}\right| ^{2}-E|\beta _{jk;sN_{r}}|^{2}\right\}
\end{equation*}%
\begin{equation*}
=\frac{1}{D(D-1)}\sum_{k}\left\{ (D-1)\sum_{r}E|\beta
_{jk;sN_{r}}|^{2}-\sum_{r_{1}\neq r_{2}}\left| \beta _{jk;sr_{1}}-\beta
_{jk;sr_{2}}\right| ^{2}\right\} ,
\end{equation*}%

and applying again the Diagram Formula, we have that%

\begin{equation*}
Var\left( \widetilde{\Gamma }_{j;s}^{CP}-\widetilde{\Gamma }%
_{j;s}^{AP}\right)
\end{equation*}%
\begin{equation*}
=\frac{1}{D^{2}(D-1)^{2}}\sum_{k_{1},k_{2}}\sum_{r_{1}\neq r_{2,}r_{3}\neq
r_{4,}}\left| E\left( \beta _{jk_{1};sr_{1}}-\beta _{jk_{1};sr_{2}}\right)
\left( \overline{\beta _{jk_{2};sr_{3}}}-\overline{\beta _{jk_{2};sr_{4}}}%
\right) \right| ^{2}.
\end{equation*}%

Similarly to the discussion for (\ref{xiao4}), we can show that%
\begin{equation*}
Var\left( \widetilde{\Gamma }_{j;s}^{CP}-\widetilde{\Gamma }%
_{j;s}^{AP}\right) =O\left( D^{2}B^{2\left( 1-\gamma )\right) j}\right) .
\end{equation*}%
Once again, the next step is to consider the fourth order cumulants,

\begin{equation*}
Cum_{4}\left\{ \sum_{k}\left( \sum_{r_{1}\neq r_{2}}\left| \beta
_{jk;sr_{1}}-\beta _{jk;sr_{2}}\right| ^{2}-(D-1)\sum_{r}E|\beta
_{jk;sN_{r}}|^{2}\right) \right\}
\end{equation*}%

\begin{equation*}
=6\sum_{k_{1},..,k_{4}}\sum_{r_{2n}\neq r_{2n-1,}n=1,..,4}E\left( \beta
_{jk_{1};sr_{1}}-\beta _{jk_{1};sr_{2}}\right) \left( \overline{\beta
_{jk_{2};sr_{3}}}-\overline{\beta _{jk_{2};sr_{4}}}\right)
\end{equation*}%
\begin{equation*}
\times E\left( \beta _{jk_{2};sr_{3}}-\beta _{jk_{2};sr_{4}}\right) \left(
\overline{\beta _{jk_{3};sr_{5}}}-\overline{\beta _{jk_{3};sr_{6}}}\right)
E\left( \beta _{jk_{3};sr_{5}}-\beta _{jk_{3};sr_{6}}\right) \left(
\overline{\beta _{jk_{4};sr_{7}}}-\overline{\beta _{jk_{4};sr_{8}}}\right)
\end{equation*}%

\begin{equation*}
\times E\left( \beta _{jk_{4};sr_{7}}-\beta _{jk_{4};sr_{8}}\right) \left(
\overline{\beta _{jk_{1};sr_{1}}}-\overline{\beta _{jk_{1};sr_{2}}}\right)
\end{equation*}%

\begin{equation*}
\leq C_{M}D^{4}\left( \Gamma _{j;s}^{N}\right) ^{4}\sum_{k_{1},..,k_{4}}%
\frac{\lambda _{jk_{1}}\lambda _{jk_{2}}\lambda _{jk_{3}}\lambda _{jk_{4}}}{%
\left[ 1+d(\xi _{jk_{1}},\xi _{jk_{2}})\right] ^{M}\left[ 1+d(\xi
_{jk_{2}},\xi _{jk_{3}})\right] ^{M}}
\end{equation*}%
\begin{equation*}
\times \frac{1}{\left[ 1+d(\xi _{jk_{3}},\xi _{jk_{4}})\right] ^{M}\left[
1+d(\xi _{jk_{4}},\xi _{jk_{1}})\right] ^{M}}
\end{equation*}%
\begin{equation*}
\leq CD^{4}B^{\left( 2-4\gamma \right) j}.
\end{equation*}%

in view of (\ref{4}), choosing $M\geq 3$. Now it is easy to see that%
\begin{equation*}
Cum_{4}\left\{ \frac{\widetilde{\Gamma }_{j;s}^{CP}-\widetilde{\Gamma }%
_{j;s}^{AP}}{\sqrt{Var\left\{ \widetilde{\Gamma }_{j;s}^{CP}-\widetilde{%
\Gamma }_{j;s}^{AP}\right\} }}\right\} \rightarrow 0\text{ },
\end{equation*}%
whence the Proposition is established, again resorting to results in (\cite%
{np})
\end{proof}

\begin{remark}
Note that $Var\left\{ \widetilde{\Gamma }_{j;s}^{CP}\right\} ,Var\left\{
\widetilde{\Gamma }_{j;s}^{AP}\right\} ,2Cov\left\{ \widetilde{\Gamma }%
_{j;s}^{CP},\widetilde{\Gamma }_{j;s}^{AP}\right\} $ are robust to
misspecification of the noise, because Variance and Covariance are
translation invariant. It follows that the denominator can (once again) be
consistently estimated by subsampling techniques, as in \cite{bkmpBer}.
\end{remark}

Under the alternative of noise misspecification, we have easily
\begin{equation*}
\frac{\widetilde{\Gamma }_{j;s}^{CP}-\widetilde{\Gamma }_{j;s}^{AP}}{\sqrt{%
Var\left\{ \widetilde{\Gamma }_{j;s}^{CP}-\widetilde{\Gamma }%
_{j;s}^{AP}\right\} }}\rightarrow _{d}N(\delta _{j},1)
\end{equation*}%
where%
\begin{equation*}
\delta _{j}:=\frac{E|\beta _{jk;sN_{r}}|^{2}-\Gamma _{j;sN_{r}}}{\sqrt{%
Var\left\{ \widetilde{\Gamma }_{j;s}^{CP}-\widetilde{\Gamma }%
_{j;s}^{AP}\right\} }}
\end{equation*}%
where $\Gamma _{j;sN_{r}}$ is the bias-correction term which is wrongly
adopted. The derivation of the power properties of this testing procedure is
then immediate.

As a final comment, we notice that throughout this paper we have only been
considering estimation and testing for the total angular power spectrum $%
C_{l}=C_{lE}+C_{lB}.$ The separate estimation of the two components ($E$ and
$B$ modes) is of great interest for physical applications, and will be
addressed in future work.

\end{document}